\documentclass[oneside, 11pt]{amsart}

\usepackage[utf8]{inputenc}

\usepackage{float}
\usepackage{hyperref}
\usepackage{geometry}
\usepackage{mathtools}
\usepackage[shortlabels]{enumitem}
\usepackage{amsmath}
\usepackage{amssymb}
\usepackage{amsthm}
\usepackage{booktabs}
\usepackage{stackengine}
\usepackage{mathrsfs}
\usepackage[capitalize]{cleveref}
\usepackage{tikz-cd}
\usepackage{microtype}
\usepackage{xspace}

\newtheorem{theorem}{Theorem}[section]
\newtheorem{lemma}[theorem]{Lemma}

\newtheorem{corollary}[theorem]{Corollary}

\theoremstyle{definition}
\newtheorem{definition}[theorem]{Definition}
\newtheorem{remark}[theorem]{Remark}

\usetikzlibrary{decorations.pathmorphing}
\tikzset{symbol/.style={draw=none,every to/.append style={
      edge node={node [sloped, allow upside down, auto=false]{$#1$}}}}}
\newcommand{\ardual}{\ar[r, <->,squiggly]}
\newcommand{\ardualadj}{\ar[r, ->,squiggly, shift right] \ar[r, <-,squiggly, shift left]}
\newcommand{\arrestrict}{\ar[u, symbol=\leqslant]}
\newcommand{\arsemirestrict}{\ar[u, symbol=\preccurlyeq]}
\newcommand{\areqv}{\ar[r, <->]}

\newcommand{\upset}{\mathord\uparrow}
\newcommand{\downset}{\mathord\downarrow}

\newcommand{\cat}[1]{{\sf #1}\xspace}
\newcommand{\Top}{\cat{Top}}
\newcommand{\Sob}{\cat{Sob}}
\newcommand{\LCSob}{\cat{LKSob}}
\newcommand{\SLCSp}{\cat{StLKSp}}
\newcommand{\SKSp}{\cat{StKSp}}
\newcommand{\KHaus}{\cat{KHaus}}

\newcommand{\Frm}{\cat{Frm}}
\newcommand{\SFrm}{\cat{SFrm}}
\newcommand{\CFrm}{\cat{ConFrm}}
\newcommand{\SCFrm}{\cat{StCFrm}}
\newcommand{\SKFrm}{\cat{StKFrm}}
\newcommand{\KRFrm}{\cat{KRFrm}}

\newcommand{\LPries}{\cat{LPries}}
\newcommand{\SL}{\cat{SLPries}}
\newcommand{\CL}{\cat{ConLPries}}
\newcommand{\SCL}{\cat{StCLPries}}
\newcommand{\SKL}{\cat{StKLPries}}
\newcommand{\KRL}{\cat{KRLPries}}

\newcommand{\clopup}{{\sf ClopUp}}

\newcommand{\functor}[1]{\mathscr #1}
\newcommand{\cl}{\mathop{\sf cl}}
\renewcommand{\int}{\mathop{\sf int}}
\DeclareMathOperator{\reg}{reg}

\setlist[enumerate,1]{label={\upshape(\arabic*)}}

\makeatletter
\@namedef{subjclassname@2020}{
  \textup{2020} Mathematics Subject Classification}
\patchcmd{\@setaddresses}{\indent}{\noindent}{}{}
\patchcmd{\@setaddresses}{\indent}{\noindent}{}{}
\patchcmd{\@setaddresses}{\indent}{\noindent}{}{}
\patchcmd{\@setaddresses}{\indent}{\noindent}{}{}
\makeatother

\title{Deriving dualities in pointfree topology from Priestley duality}

\author{G.~Bezhanishvili and S.~Melzer}
\address{\newline
Department of Mathematical Sciences\newline
New Mexico State University\newline
Las Cruces, NM 88003\newline
USA\newline}

\email{guram@nmsu.edu}
\email{smelzer@nmsu.edu}

\begin{document}
\begin{abstract}
There are several prominent duality results in pointfree topology.
Hofmann-Lawson duality establishes that the category of continuous frames is dually equivalent to the category of locally compact sober spaces.
This restricts to a dual equivalence between the categories of stably continuous frames and stably locally compact spaces, which further restricts to Isbell duality between the categories of compact regular frames and compact Hausdorff spaces.
We show how to derive these dualities from Priestley duality for distributive lattices, thus shedding new light on these classic results.
\end{abstract}

\maketitle

\tableofcontents

\section{Introduction}
In pointfree topology there is a well-known dual adjunction between the category \Top of topological spaces and continuous maps and the category \Frm of frames and frame homomorphisms (see, e.g., \cite{DowkerPapert1966}).
Let \Sob be the full subcategory of \Top consisting of sober spaces and \SFrm the full subcategory of \Frm consisting of spatial frames.
The dual adjunction between \Top and \Frm then restricts to a dual equivalence between \Sob and \SFrm (see, e.g., \cite[Sec.~II-1]{Johnstone1982}).
Further restrictions yield the following classic results:

\begin{itemize}
    \item Hofmann-Lawson duality between the category \CFrm of continuous frames and proper frame homomorphisms and the category \LCSob of locally compact sober spaces and proper maps \cite{HofmannLawson1978}.
    \item A dual equivalence between the full subcategory \SCFrm of \CFrm consisting of stably continuous frames and the full subcategory \SLCSp of \LCSob consisting of stably locally compact spaces, which further restricts to a dual equivalence between the full subcategories \SKFrm of stably compact frames and \SKSp of stably compact spaces \cite{GierzKeimel1977, Johnstone1981, Simmons1982, Banaschewski1981}.
    \item Isbell duality between the full subcategory \KRFrm of \Frm consisting of compact regular frames and the full subcategory \KHaus of \Top consisting of compact Hausdorff spaces \cite{Isbell1972}. 
\end{itemize}

Note that every frame homomorphism between compact regular frames is proper, and hence \KRFrm is a full subcategory of \SKFrm.
Similarly, \KHaus is a full subcategory of \SKSp. 
We thus arrive at the following diagram, where a pair of squiggly arrows (\stackanchor{\reflectbox{$\rightsquigarrow$}}{$\rightsquigarrow$}) represents a dual adjunction and a squiggly left-right arrow ($\leftrightsquigarrow$) a dual equivalence. Also, $\cat{C}\leqslant \cat{D}$ stands for ``\cat{C} is a full subcategory of \cat{D}'' and $\cat{C} \preccurlyeq\cat{D}$ for ``\cat{C} is a non-full subcategory of \cat{D}.''

\begin{figure}[H]
    \begin{center}
        \begin{tikzcd}[ampersand replacement=\&, column sep={8em,between origins}, row sep=1em]
            \Frm \ardualadj \& \Top\\
            \SFrm \arrestrict \ardual \& \Sob \arrestrict\\
            \CFrm \arsemirestrict \ardual \& \LCSob \arsemirestrict\\
            \SCFrm  \ardual\arrestrict \& \SLCSp \arrestrict\\
            \SKFrm  \ardual\arrestrict \& \SKSp\arrestrict\\
            \KRFrm  \ardual\arrestrict \& \KHaus\arrestrict
        \end{tikzcd}
    \end{center}
    \caption{Correspondence between various categories of frames and spaces.} \label{diagram: introduction}
\end{figure}
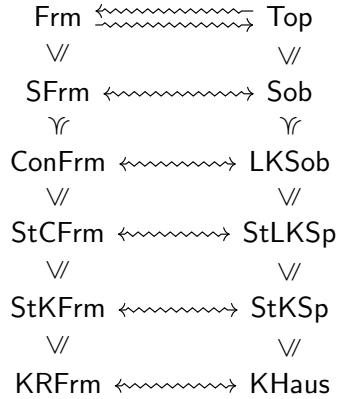

It is our aim to provide a different perspective on these dualities by utilizing Priestley duality \cite{Priestley1970,Priestley1972}, which establishes a dual equivalence between the categories \cat{DLat} of bounded distributive lattices and bounded lattice homomorphisms and \cat{Pries} of Priestley spaces and Priestley morphisms. 

It is well known (see, e.g., \cite[Sec.~II-3]{Johnstone1982}) that \cat{DLat} is equivalent to the category \cat{CohFrm} of coherent frames, which is dually equivalent to the category \cat{Spec} of spectral spaces.
We recall that a stably compact frame $L$ is {\em coherent} if compact elements join-generate $L$, and a stably compact space $X$ is {\em spectral} if compact opens form a basis for $X$.
Therefore, \cat{CohFrm} is a full subcategory of \SKFrm and \cat{Spec} is a full subcategory of \SKSp. By \cite{Cornish1975}, \cat{Spec} is isomorphic to \cat{Pries}.
Thus, Priestley duality can be derived from the dual equivalence of \cat{CohFrm} and \cat{Spec}.

On the other hand, since frames are special distributive lattices, they can be studied using the machinery of Priestley duality.
This line of research was initiated by Pultr and Sichler \cite{PultrSichler1988} who showed that Priestley duality restricts to a dual equivalence between \Frm and the category \LPries of what we call L-spaces and L-morphisms (see \cref{sec: 3}).
It was further developed in \cite{PultrSichler2000, BezhGhilardi2007, BezhanishviliGabelaiaJibladze2016, AvilaBezhanishviliMorandiZaldivar2020, AvilaBezhanishviliMorandiZaldivar2021} where various properties of frames were characterized in the language of their Priestley duals.
An alternative approach, using \cat{Spec} instead of \cat{Pries}, was investigated in \cite{Schwartz2013,DickmannSchwartzTressl2019}.

The exploration of Priestley spaces of frames has numerous applications, not only in pointfree topology, but in other areas as well.
For example, nuclei play an important role in pointfree topology as they are kernels of frame homomorphisms (see, e.g., \cite[p.~31]{PicadoPultr2012}).
But they also arise in logic as they model the so-called lax modality \cite{Goldblatt1981,FairtloughMendler1997}.
The resulting intuitionistic modal logic has various applications \cite{FairtloughMendler1995,AlechinaMendlerDePaiva2001,GargAbadi2008,Goldblatt2011,ArtemovProtopopescu2016}.
As was demonstrated in \cite{BezhanishviliHolliday2019}, nuclei also provide a unified semantic hierarchy for intuitionistic logic.
It turns out that nuclei have a rather natural description in the language of Priestley spaces, which has resulted in numerous insides in understanding the complicated structure of the frame of nuclei of a given frame (see, e.g., \cite{PultrSichler2000, BezhGhilardi2007, BezhanishviliGabelaiaJibladze2013, AvilaBezhanishviliMorandiZaldivar2020, AvilaBezhanishviliMorandiZaldivar2021}).

The goal of this paper is to continue the study of frames by means of their Priestley spaces.
In particular, we provide dual descriptions of the categories \SFrm, \CFrm, \SCFrm, \SKFrm, and \KRFrm in the language of Priestley spaces.
On the one hand, this yields an alternative proof of the dual equivalences mentioned at the beginning of the introduction, thus providing a new insight into these classic results in pointfree topology from the perspective of Priestley duality.
On the other hand, it gives rise to new subcategories of Priestley spaces that are equivalent to such important categories of topological spaces as \Sob, \LCSob, \SLCSp, \SKSp, and \KHaus.
It is our belief that results of this nature can provide further insight and cross-fertilization between these beautiful branches of mathematics. 

The paper is organized as follows.
\cref{sec: frm} introduces the categories of frames and spaces of interest, and presents the relevant dualities in more detail.
\cref{sec: 3} discusses Priestley duality and its restriction to frames.
\cref{sec: 3.5} characterizes spatial frames in the language of Priestley duality and connects the associated Priestley spaces with sober spaces.
\cref{sec: 4} further restricts this correspondence to continuous frames, their associated Priestley spaces, and locally compact sober spaces. This yields a new proof of Hofmann-Lawson duality.
In \cref{sec: 5} we derive the duality between stably continuous frames and stably locally compact spaces by describing stability in the language of Priestley spaces.
This also gives a new proof of the duality between stably compact frames and stably 
compact spaces.
Finally, \cref{sec: 6} describes regularity in the language of Priestley spaces, thus providing an 
alternative proof of Isbell duality. 

\section{Frames and spaces} \label{sec: frm}
We recall (see, e.g., \cite[p.~10]{PicadoPultr2012}) that a {\em frame} is a complete lattice $L$ satisfying the join-infinite distributive law 
\begin{equation*}
    a \wedge \bigvee S = \bigvee \{a \wedge s \mid s \in S\}
\end{equation*}
for every $a \in L$ and $S \subseteq L$. A {\em frame homomorphism} is a map between frames that preserves finite 
meets and arbitrary joins.
Let \Frm be the category of frames and frame homomorphisms.

A filter $F$ of a frame $L$ is \emph{completely prime} if $\bigvee S \in F$ implies $S \cap F \neq \varnothing$ and \emph{Scott-open} if $\bigvee S \in F$ implies $\bigvee T \in F$ for some finite $T \subseteq S$.
Clearly each completely prime filter is Scott-open. In fact, Scott-open filters are exactly the intersections of completely prime filters (see, e.g., \cite[p.~101]{Vickers1989}). 

A frame $L$ is \emph{spatial} if completely prime filters separate elements of $L$; that is, $a \not \leq b$ in $L$ implies that there is a completely prime filter $F$ with $a \in F$ and $b \not \in F$.
Equivalently, $L$ is spatial iff Scott-open filters separate elements of $L$.
It is well known (see, e.g., \cite[p.~18]{PicadoPultr2012}) that $L$ is spatial iff $L$ is isomorphic to the frame $\mathcal O(X)$ of open sets of a topological space $X$ (hence the name). 
Let \SFrm be the full subcategory of \Frm consisting of spatial frames.

There are two relations on frames that are of particular importance to us. We recall that if $L$ is a frame and $a\in L$, then the {\em pseudocomplement} of $a$ is $a^*=\bigvee\{x \in L \mid a\wedge x=0\}$. 

\begin{definition}
    Let $L$ be a frame and $a,b \in L$.
    \begin{enumerate}
        \item (see, e.g., \cite[p.~49]{Compendium2003}) We say that $a$ is {\em way below} $b$ and write $a\ll b$ if for each $S\subseteq L$, from $b\le \bigvee S$ it follows that $a\le\bigvee T$ for some finite $T\subseteq S$.
        \item (see, e.g., \cite[p.~80]{Johnstone1982}) We say that $a$ is {\em well inside} $b$ and write $a\prec b$ if $a^*\vee b=1$.
    \end{enumerate} 
\end{definition}

A frame $L$ is {\em continuous} if 
\[
    a = \bigvee\{b \in L \mid b \ll a\}
\] 
and $L$ is {\em regular} if
\[
    a = \bigvee\{b \in L \mid b \prec a\}
\]
for all $a \in L$. 

Each frame homomorphism $h:L\to M$ preserves $\prec$ (that is, $a\prec b$ implies $h(a) \prec h(b)$), but may not 
preserve $\ll$. We call $h$ {\em proper} provided $h$ preserves $\ll$ (that is, $a\ll b$ implies 
$h(a) \ll h(b)$).

\begin{definition}
    Let \CFrm be the category of continuous frames and proper frame homomorphisms between them.  
\end{definition}

We call an element $a$ of a frame $L$ {\em compact} if $a\ll a$ and the frame $L$ {\em compact} if the top element $1$ 
is compact. We also call $\ll$ {\em stable} if $a \ll b,c$ implies $a \ll b \wedge c$ for all $a,b,c \in L$. 
\begin{definition}
    \begin{enumerate}
        \item[]
        \item (see, e.g., \cite[p.~488]{Compendium2003}) A frame $L$ is {\em stably continuous} if $L$ is continuous 
        and $\ll$ is stable. Let \SCFrm be the full subcategory of \CFrm consisting of stably continuous frames.
        \item (see, e.g., \cite[p.~488]{Compendium2003}) A frame $L$ is {\em stably compact} if $L$ is compact and 
        stably continuous. Let \SKFrm be the full subcategory of \SCFrm consisting of stably compact frames.
        \item (see, e.g., \cite[p.~133]{PicadoPultr2012}) Let \KRFrm be the full subcategory of \Frm consisting of 
        compact regular frames.
    \end{enumerate}
\end{definition}

We note that if $L$ is compact, then $a\prec b$ implies $a\ll b$, and if $L$ is regular, then $a\ll b$ implies 
$a\prec b$ (see, e.g., \cite[Lem.~VII-5.2.1]{PicadoPultr2012}). Therefore, if $L$ is compact regular, then the way below 
and well inside relations on $L$ coincide. Thus, since $a\prec b,c$ implies $a\prec b\wedge c$, every compact regular 
frame is stably compact. Consequently, \KRFrm is a full subcategory of \SKFrm.

The next table contains the categories of frames that we will be concerned with in this paper.

\begin{table}[H]
    \begin{tabular}{lll}
        \toprule 
        \bf Category & \bf Objects & \bf Morphisms \\ \midrule
        \Frm & Frames & Frame homomorphisms\\
        \SFrm& Spatial frames & Frame homomorphisms\\
        \CFrm  & Continuous frames        & Proper frame homomorphisms \\
        \SCFrm & Stably continuous frames & Proper frame homomorphisms \\
        \SKFrm & Stably compact frames    & Proper frame homomorphisms \\ 
        \KRFrm & Compact regular frames   & Frame homomorphisms \\
        \bottomrule
    \end{tabular}
    \caption{Categories of frames.\label{table:frames}}
\end{table}

We next turn our attention to the categories of spaces that correspond to the above categories of frames. The 
definitions that follow are well known; see, e.g., \cite{Compendium2003}.

For a partially ordered set $P$ and $S \subseteq P$, we write 
\begin{center}
    $\upset S = \{x \in P \mid s \leq x \text{ for some } s \in S\}$ 
    \quad and \quad 
    $\downset S = \{x \in P \mid x \leq s \text{ for some } s \in S\}$.
\end{center} 
Then $S$ is an \emph{upset} if $S = \upset S$, $S$ is a \emph{downset} if $S = \downset S$, and $S$ is a \emph{biset} 
if it is both an upset and a downset. For a singleton $S = \{x\}$ we write $\upset x$ for $\upset S$ and $\downset x$ 
for $\downset S$. 

Let $X$ be a topological space. We recall that a closed subset $A$ of $X$ is {\em irreducible} if $A$ cannot be written 
as a union of two closed proper subsets, and that $X$ is {\em sober} if each irreducible subset of $X$ is the closure 
of a unique point in $X$. In particular, every sober space is $T_0$. 

The space $X$ is \emph{locally compact} if for each open set $U$ and $x \in U$ there is an open set $V$ and a compact 
set $K$ such that $x \in V \subseteq K \subseteq U$. A subset of $X$ is {\em saturated} if it is an intersection of 
open sets, and $X$ is {\em coherent} if the intersection of two compact saturated sets is again compact. 

The \emph{specialization preorder} $\leq$ on $X$ is defined by $x \leq y$ iff $x \in \cl(\{y\})$, where $\cl$ is 
topological closure. Observe that saturated sets of $X$ are exactly upsets in the specialization preorder, which is a 
partial order iff $X$ is a $T_0$-space. 

\begin{definition}
    \begin{enumerate}[ref=\thedefinition(\arabic*)]
        \item[]
        \item Let \Top be the category of topological spaces and continuous maps, and let \Sob be the full subcategory 
        of \Top consisting of sober spaces. 
        \item A continuous map $f : X \to Y$ between topological spaces is {\em proper} if 
        \label[definition]{def:2.4(2)}
        \begin{enumerate}
            \item[(i)] $\downset f(A)$ is closed for each closed set $A \subseteq X$, where $\downset$ is the downset 
            in the specialization preorder on $X$. 
            \item[(ii)] $f^{-1}(B)$ is compact for each compact saturated set $B \subseteq Y$.
        \end{enumerate}
        \item Let \LCSob be the category of locally compact sober spaces and proper maps between them.
        \item We call $X$ {\em stably locally compact} if $X$ is locally compact, sober, and coherent. Let \SLCSp be 
        the full subcategory of \LCSob consisting of stably locally compact spaces. 
        \item We call $X$ {\em stably compact} if $X$ is compact and stably locally compact. Let \SKSp be the full 
        subcategory of \SLCSp consisting of stably compact spaces.
        \item Let \KHaus be the full subcategory of \Sob consisting of compact Hausdorff spaces.
    \end{enumerate}
\end{definition}

\begin{remark}
    \begin{enumerate}[ref=\theremark(\arabic*)]
        \item[]
        \item By \cite[Lem.~VI-6.21]{Compendium2003}, if $X$ is sober and $Y$ is locally compact, then (i) follows 
        from (ii) in \cref{def:2.4(2)}. \label[remark]{rem:2.5}
        \item In compact Hausdorff spaces, the specialization order is the identity. Hence, compact saturated sets are 
        simply closed sets. Therefore, since every compact Hausdorff space is sober and locally compact, \KHaus is a 
        full subcategory of \SKSp. 
    \end{enumerate}
\end{remark}

The next table contains the categories of topological spaces that we are interested in.

\begin{table}[H]
\begin{tabular}{lll}
    \toprule 
    \bf Category & \bf Objects & \bf Morphisms \\ \midrule
    \Top & Topological spaces & Continuous maps\\
    \Sob & Sober spaces & Continuous maps\\
    \LCSob & Locally compact sober spaces  & Proper maps \\
    \SLCSp & Stably locally compact spaces & Proper maps \\
    \SKSp  & Stably compact spaces         & Proper maps \\ 
    \KHaus & Compact Hausdorff spaces      & Continuous maps\\
    \bottomrule
\end{tabular}
\caption{Categories of spaces.\label{table:spaces}}
\end{table}

There is a well-known dual adjunction $(\mathcal O, pt)$ between \Top and \Frm (see, e.g., \cite{DowkerPapert1966}). 
The contravariant functors $\mathcal O : \Top \to \Frm$ and $pt : \Frm \to \Top$ are constructed as follows. The functor 
$\mathcal O$ maps a topological space $X$ to its frame of open sets, and a continuous map $f : X \to Y$ to 
$f^{-1} : \mathcal O(Y)\to\mathcal O(X)$. The functor $pt$ maps a frame $L$ to its space of points, where a 
\emph{point} is a completely prime filter of $L$. The topology on $pt(L)$ is the range of the map 
$\zeta : L \to \wp(pt (L))$ given by $\zeta(a) = \{x \in pt(L) \mid a \in x\}$, where $\wp(X)$ denotes the 
powerset of a set $X$. A frame homomorphism $h : L \to M$ is mapped to $h^{-1} : pt(M) \to pt(L)$. 

Restricting the range of these functors yields the following well-known duality results:

\begin{theorem} \label{thm:dualities}
    \begin{enumerate}[ref=\thetheorem(\arabic*)]
        \item[]
        \item {\SFrm} is dually equivalent to {\Sob}. \label[theorem]{SFrm-deqv-Sob}
        \item {\em (Hofmann-Lawson duality)} {\CFrm} is dually equivalent to {\LCSob}. 
        \item {\SCFrm} is dually equivalent to {\SLCSp}. \label[theorem]{SCFrm-deqv-SLCsp}
        \item {\SKFrm} is dually equivalent to {\SKSp}. \label[theorem]{SKFrm-deqv-SKsp}
        \item {\em (Isbell duality)} {\KRFrm} is dually equivalent to {\KHaus}.
    \end{enumerate}
\end{theorem}

For \cref{SFrm-deqv-Sob} see, e.g., \cite[Sec.~II-1]{Johnstone1982}. Hofmann-Lawson duality was established in 
\cite{HofmannLawson1978} (see also \cite[Prop.~V-5.20]{Compendium2003}). \cref{SCFrm-deqv-SLCsp,SKFrm-deqv-SKsp} go 
back to \cite{GierzKeimel1977, Johnstone1981, Simmons1982, Banaschewski1981} (see also 
\cite[Thm.~VI-7.4]{Compendium2003}). Isbell duality was established in \cite{Isbell1972} (see also 
\cite{BanaschweskiMulvey1980} and \cite[Sec.~VII-4]{Johnstone1982}). We thus obtain the diagram in 
\cref{diagram: introduction}.

\section{Priestley duality for frames} \label{sec: 3}

As is customary, we call a subset of a topological space $X$ \emph{clopen} if it is both closed and open. Then $X$ is 
\emph{zero-dimensional} if it has a basis of clopen sets. {\em Stone spaces} are zero-dimensional compact Hausdorff 
spaces. 

\begin{definition}
    A \emph{Priestley space} is a pair $(X, \leq)$ where $X$ is a compact space and $\leq$ is a partial order on $X$ 
    satisfying the \emph{Priestley separation axiom}:
    \begin{center}
        If $x\nleq y$, then there is a clopen upset $U$ with $x \in U$ and $y \not \in U$.
    \end{center}
\end{definition}

For a Priestley space $(X, \leq)$ we simply write $X$ and note that every Priestley space is a Stone space.

A \emph{Priestley morphism} is a continuous map $f : X \to Y$ between Priestley spaces that is order-preserving. Let 
\cat{Pries} be the category of Priestley spaces and Priestley morphisms. Let also \cat{DLat} by the category of bounded 
distributive lattices and bounded lattice homomorphisms.

\begin{theorem}[Priestley duality] \label{thm: PD}
     \cat{Pries} is dually equivalent to \cat{DLat}.
\end{theorem}

\begin{remark}
    \begin{enumerate}[(1), ref=\theremark(\arabic*)]
        \item[]
        \item For a bounded distributive lattice $D$, the \emph{Priestley space} $X_D$ of $D$ is given by the set of
        prime filters of $D$ ordered by inclusion and topologized by the subbasis 
        $$\{\varphi(a) \mid a \in D\} \cup \{\varphi(a)^c \mid a \in D\},$$
        where $\varphi : D \to \wp(X_D)$ is the Stone map given by $\varphi(a) = \{x \in X_D \mid a \in x\}$ for each 
        $a \in D$.
        \item \label[remark]{rem: X and D} The contravariant functors $\functor X : \cat{DLat} \to \cat{Pries}$ and 
        $\functor D : \cat{Pries} \to \cat{DLat}$ establishing Priestley duality are described as follows. The functor 
        $\functor X$ sends a bounded distributive lattice $D$ to the Priestley space $X_D$ of $D$ and a bounded lattice 
        homomorphism $h : D \to E$ to the Priestley morphism $h^{-1} : X_E \to X_D$. The functor $\functor D$ sends a 
        Priestley space $X$ to the bounded distributive lattice $\clopup(X)$ of clopen upsets of $X$ and a Priestley 
        morphism $f : X \to Y$ to the bounded lattice homomorphism $f^{-1} : \clopup(Y) \to \clopup(X)$. The natural 
        isomorphisms are given by $\varphi : D \to \functor{D}\functor{X} D$ defined above and 
        $\varepsilon : X \to \functor{X}\functor{D} X$ defined by $\varepsilon(x) = \{U \in \clopup(X) \mid x \in U\}$ for 
        each $x \in X$.
    \end{enumerate}
\end{remark}

Since frames are special bounded distributive lattices, we can restrict Priestley duality to obtain a category that is 
dual to \Frm. It is well known that a bounded distributive lattice is a frame iff it is a complete Heyting algebra 
(see, e.g., \cite[Prop.~1.5.4]{Esakia2019}). Therefore, we can describe the dual category of \Frm using Esakia duality 
\cite{Esakia1974}. We recall that a Priestley space is an \emph{Esakia space} if $U$ clopen implies that $\downset U$ 
is clopen. The following are well-known properties of Priestley and Esakia spaces that we will use frequently. For a 
subset $F$ of a poset $X$, we write $\min(F)$ and $\max(F)$ for the sets of minimal and maximal elements of $F$, 
respectively.

\begin{lemma}[see, e.g., \cite{Priestley1984,Esakia2019}] \label{lemma:properties-Priestley}
    Let $X$ be a Priestley space.
\begin{enumerate}[ref=\thetheorem(\arabic*)]
\item The collection of clopen upsets and clopen downsets of $X$ forms a subbasis for $X$. \label[lemma]{lemma:properties-Priestley-0}
    \item Every open upset/downset is a union of clopen upsets/downsets. \label[lemma]{lemma:properties-Priestley-1}
    \item Every closed upset/downset is an intersection of clopen upsets/downsets. \label[lemma]{lemma:properties-Priestley-2}
    \item \label[lemma]{lemma:properties-Priestley-2.5} If $F \subseteq X$ is closed, then both $\upset F$ and $\downset F$ are closed.
    \item \label[lemma]{lemma:properties-Priestley-3} If $F$ is a closed set, then $\min(F)$ and $\max(F)$ are nonempty. In fact, for every $x \in F$ there exist $y \in \min(F)$ and $z \in \max(F)$ such that $y \leq x \leq z$. Consequently, if $F$ is a closed upset, then $F = \upset \min(F)$ and if $F$ is a closed downset, then $F = {\downarrow} \max(F)$.
    \item If $\mathcal P$ is a prime filter of  $\clopup(X)$, then $\bigcap \mathcal P = \upset x$ for a unique $x \in X$. \label[lemma]{lemma:properties-Priestley-4}
\end{enumerate}
Suppose additionally that $X$ is an Esakia space.
\begin{enumerate}[ref=\thetheorem(\arabic*), resume*]
    \item $\cl \upset F = \upset \cl F$. Consequently, the closure of an upset is an upset and the interior of a downset is a downset. \label[lemma]{lem:esakia-int-downsets}
\end{enumerate}
\end{lemma}

Let $D\in\cat{DLat}$ and $X_D$ be its Priestley space. Then $D$ is a Heyting algebra iff $X_D$ is an Esakia space \cite{Esakia1974}, and $D$ is a complete Heyting algebra iff $X_D$ is an extremally order-disconnected Esakia space \cite[Thm.~2.4(2)]{BezhanishviliBezhanishvili2008}, where we recall that an Esakia space is \emph{extremally order-disconnected} if $\cl U$ is open for every open upset $U$. We thus arrive at the following well-known result. It was first proved in \cite[Thm.~2.3]{PultrSichler1988} without using Esakia duality. For the formulation below, see \cite[Thm.~3.7]{AvilaBezhanishviliMorandiZaldivar2020}.

\begin{theorem}
    Let $D$ be a bounded distributive lattice and $X_D$ its Priestley space. Then $D$ is a frame iff $X_D$ is an extremally order-disconnected Esakia space.
\end{theorem}

The following well-known fact (see, e.g., \cite[Lem.~2.3]{BezhanishviliBezhanishvili2008}) will be used throughout. 

\begin{lemma} \label{lemma:joins-in-priestley}
    For a frame $L$, its Priestley space $X_L$, and $S \subseteq L$, we have \begin{equation*}
        \varphi \left( \bigvee S \right) = \cl \left(\bigcup \{\varphi(s) \mid s \in S\}\right).
    \end{equation*}
\end{lemma}

Frame homomorphisms are dually characterized by Priestley morphisms that satisfy the following additional condition.

\begin{lemma} [{\cite[Sec.~2.5]{PultrSichler1988}}]
    Let $L,M\in\Frm$, $h: L \to M$ be a bounded lattice homomorphism, and $f=\functor X(h)$.
    Then $h$ is a frame homomorphism iff $f^{-1}\cl U = \cl f^{-1}(U)$ for all open upsets $U$ of $X_L$.
\end{lemma}

Since frames are also known as locales (see, e.g., \cite{PicadoPultr2012}), we introduce the following terminology.

\begin{definition}[L-spaces]  \label{def: L}
\begin{enumerate}[ref=\thedefinition(\arabic*)]
\item[] 
\item A {\em localic space} or simply an \emph{L-space} is an extremally order-disconnected Esakia space. 
\item An \emph{L-morphism} is a Priestley morphism $f : X \to Y$ between L-spaces such that $f^{-1}\cl U = \cl f^{-1}(U)$ for all open upsets $U$ of $Y$. \label[definition]{def:l-morph}
\item Let {\LPries} be the category of L-spaces and L-morphisms.
\end{enumerate}
\end{definition}

\begin{theorem}[Pultr-Sichler {\cite[Cor.~2.5]{PultrSichler1988}}]  \label{thm: PS}
    \Frm is dually equivalent to \LPries.
\end{theorem}

\begin{remark} \label{rem: units}
The functors establishing Pultr-Sichler duality are the restrictions of the functors $\functor X : \cat{DLat} \to \cat{Pries}$ and $\functor D : \cat{Pries} \to \cat{DLat}$ establishing Priestley duality, and the units of this duality are the restrictions of the units $\varphi$ and $\varepsilon$ of Priestley duality (see \cref{rem: X and D}).
\end{remark}

\section{Priestley duality for spatial frames} \label{sec: 3.5}

Let $L$ be a frame and $X_L$ the Priestley space of $L$. Since completely prime filters are prime filters, $pt(L)$ is a subset of $X_L$, which from now on will be denoted by $Y_L$. In \cite{PultrSichler2000} elements of $Y_L$ are called \emph{L-points} and in \cite{AvilaBezhanishviliMorandiZaldivar2020} they are called \emph{nuclear points}. We follow the terminology of \cite{Schwartz2013} 
and call them \emph{localic points}. In addition, we refer to $Y_L$ as the {\em localic part} of $X_L$.
The next lemma shows that open subsets of $Y_L$ are exactly the intersections of clopen upsets of $X_L$ with $Y_L$.

\begin{lemma}[{\cite[Lem.~5.3(1)]{AvilaBezhanishviliMorandiZaldivar2020}}] \label{lemma:zeta-is-phi-restricted}
    Let $L$ be a frame, $X_L$ its Priestley space, and $Y_L \subseteq X_L$ the 
        localic part of $X_L$. Then $\zeta(a) = \varphi(a) \cap Y_L$ for each $a\in L$.
\end{lemma}

The following characterization of $Y_L$ was given in \cite[Prop.~2.9]{PultrSichler2000}; see also \cite[Lem.~5.1]{BezhanishviliGabelaiaJibladze2016}.

\begin{lemma} \label{lemma:downset-y-clopen}
Let $L$ be a frame, $X_L$ its Priestley space, $Y_L \subseteq X_L$ the localic part of $X_L$, and $x\in X_L$. Then $x \in Y_L$ iff $\downset x$ is clopen.
\end{lemma}

This motivates the following definition.

\begin{definition} \label{def: localic part}
    Let $X$ be an L-space. We call 
    \[Y := \{y \in X \mid \downset y \text{ is clopen}\}\]
    the \emph{localic part of $X$}. We view $Y$ as a topological space, where $U \subseteq Y$ is open iff $U = V \cap Y$ for some $V \in \clopup(X)$.
\end{definition}

\begin{definition}
    Let $X$ be an $L$-space and $Y$ the localic part of $X$.
We call a closed upset $F$ of $X$ a \emph{Scott upset} if $\min(F) \subseteq Y$.
\end{definition}

Scott upsets were introduced in \cite{BezhanishviliMelzer2022} where it was shown that they provide a characterization of Scott-open filters of a frame in the language of Priestley spaces. The next lemma provides a characterization of Scott upsets.

\begin{lemma}[{\cite[Lem.~5.1]{BezhanishviliMelzer2022}}] \label{lem: Scott upset}
    Let $X$ be an L-space and $F$ a closed upset of $X$. Then
    $F$ is a Scott upset iff for every open upset $U$ of $X$, from $F \subseteq \cl U$ it follows that $F \subseteq U$. 
\end{lemma}

\begin{remark}
In \cite{PultrSichler2000} closed sets satisfying the property in Lemma~\ref{lem: Scott upset} are called \emph{L-compact sets}. Thus, Scott upsets are exactly the upsets of L-compact sets.
\end{remark}

Spatial frames are characterized by the following theorem. The equivalence (1)$\Leftrightarrow$(2) is proved in \cite[Thm.~5.5]{AvilaBezhanishviliMorandiZaldivar2020} and the equivalence (1)$\Leftrightarrow$(3) in \cite[Sec.~2.11]{PultrSichler2000}.

\begin{theorem} \label{thm: spatial}
    Let $L$ be a frame, $X_L$ its Priestley space, and $Y_L \subseteq X_L$ the localic part of $X_L$. The following are equivalent. 
    \begin{enumerate}[ref=\thetheorem(\arabic*)]
        \item $L$ is spatial.
        \item $Y_L$ is dense in $X_L$. \label[theorem]{Y-dense-in-X}
        \item For clopen upsets $U,V$ of $X_L$, from $U \not \subseteq V$ it follows that there is a Scott upset $F$ of $X_L$ such that $F \subseteq U$ but $F \not \subseteq V$.
    \end{enumerate}
\end{theorem}

\begin{remark} \label{remark:spatial}
    By \cref{Y-dense-in-X}, if $U$ is clopen in $X_L$, then $\cl(U \cap Y_L) = U$. In particular, for $a \in L$, by \cref{lemma:zeta-is-phi-restricted} we have $\cl \zeta(a) = \cl(\varphi(a) \cap Y_L) = \varphi(a)$ (see also \cite[Sec.~2.12]{PultrSichler2000}).
\end{remark}

\begin{definition} \label{def: SL}
\begin{enumerate}
    \item[]
    \item An L-space $X$ is \emph{L-spatial} or simply an \emph{SL-space} if the localic part $Y$ of $X$ is dense in $X$.
    \item Let \SL be the full subcategory of \LPries consisting of SL-spaces. 
\end{enumerate}
\end{definition}

As a consequence of Theorems~\ref{thm: PS} and \ref{thm: spatial} we obtain:

\begin{corollary} \label{thm: sfrm=sl}
    \SFrm is dually equivalent to \SL.
\end{corollary}

We next connect \SL with \Sob. In order to do so, we show that mapping an L-space to its localic part is functorial. For this we need the following lemmas.

\begin{lemma} \label{lem: Y is sober}
Let $X$ be an L-space and $Y$ the localic part of $X$. Then $Y$ is a sober space.
\end{lemma}

\begin{proof}
    By \cref{lemma:zeta-is-phi-restricted,lemma:downset-y-clopen}, $Y$ is homeomorphic to $pt(\clopup(X))$. Thus, $Y$ is sober (see, e.g., \cite[p.~20]{PicadoPultr2012}).
\end{proof}

\begin{lemma}
    Let $X_1$, $X_2$ be L-spaces, $Y_1, Y_2$ their localic parts, and $f : X_1 \to X_2$ an L-morphism.
    \begin{enumerate}[ref=\thelemma(\arabic*)]
        \item $f(Y_1) \subseteq Y_2$. \label[lemma]{lem: f restricts}
        \item The restriction $f : Y_1 \to Y_2$ is a well-defined continuous map. \label[lemma]{lem: f restricts to cts}
    \end{enumerate}
\end{lemma}

\begin{proof}
   (1)  Let $y \in Y_1$ and set $U = (\downset f(y))^c$. Since $y \not \in f^{-1}(U)$ and $f^{-1}(U)$ is an upset, $\downset y \cap f^{-1}(U) = \varnothing$. Because $y\in Y_1$, we have $\downset y$ is open, so $y \not\in \cl f^{-1}(U) = f^{-1}(\cl U)$  (see \cref{def:l-morph}). 
    Therefore, $f(y) \not \in \cl U $, and so $f(y) \in \int\downset f(y)$.
    By \cref{lem:esakia-int-downsets}, $\int \downset f(y)$ is a downset. 
    Thus, $f(y) \in \int\downset f(y)$ implies that $\downset f(y) = \int\downset f(y)$, hence $\downset f(y)$ is open. Consequently, $f(y)\in Y_2$.
    
    (2) That the restriction of $f$ is well defined follows from (1). For continuity, it suffices to show that $f^{-1}(U \cap Y_2) \cap Y_1$ is open in $Y_1$ for every clopen upset $U$ of $X_2$. By (1), $f^{-1}(U \cap Y_2) \cap Y_1 = f^{-1}(U) \cap Y_1$. Since $f$ is a Priestley morphism, $f^{-1}(U)$ is a clopen upset of $X_1$. Thus, $f^{-1}(U) \cap Y_1$ is an open subset of $Y_1$. 
\end{proof}

We define a functor $\functor Y : \SL \to \Sob$ by sending an SL-space $X$ to its localic part $Y$, and an L-morphism $f:X_1\to X_2$ to its restriction $f:Y_1\to Y_2$.
It follows easily from Lemmas~\ref{lem: Y is sober} and~\ref{lem: f restricts to cts} that $\functor Y$ is a well-defined covariant functor. 

\begin{theorem} \label{thm: Y is functor}
    $\functor Y$ is essentially surjective.
\end{theorem}

\begin{proof}
    Suppose $Z$ is a sober space. Then $Z$ is homeomorphic to $pt(\mathcal O(Z))$ (see, e.g., \cite[p.~20]{PicadoPultr2012}). Let $X$ be the Priestley space of $\mathcal O(Z)$. It follows from \cref{lemma:zeta-is-phi-restricted} that $pt(\mathcal O(Z))$ is (homeomorphic to) the localic part of $X$.
\end{proof}

To show that $\functor Y$ is full and faithful, we need the following lemmas.

\begin{lemma} \label{cor:y-pulls-closure}
    Let $X$ be an L-space and $Y$ the localic part of $X$. 
    \begin{enumerate}[ref=\thelemma(\arabic*)]
    \item $\cl U \cap Y = U \cap Y$ for each open upset $U$ of $X$. \label[lemma]{cor:y-pulls-closure-1}
    \item $\cl U \cap Y = U$ for each open set $U$ of $Y$. \label[lemma]{cor:y-pulls-closure-2}
\end{enumerate}
\end{lemma}

\begin{proof}
    (1) We clearly have that $U \cap Y\subseteq\cl U \cap Y$. For the reverse inclusion, let $y \in \cl U \cap Y$. Since $y \in Y$, we have that $\downset y$ is open. Hence, $\downset y \cap U \neq \varnothing$. Thus, there is $x \in U$ with $x \leq y$. Since $U$ is an upset, we must have $y \in U$, so $y \in U \cap Y$. 
    
    (2) Since $U$ is open in $Y$, there is a clopen upset $V$ of $X$ such that $U = V \cap Y$. Thus, $\cl U \subseteq V$, and hence $U\subseteq\cl U \cap Y\subseteq V \cap Y = U$. 
\end{proof}

\begin{lemma} \label{lem:cl-respects-meet-2}
Let $X$ be an SL-space and $Y$ the localic part of $X$. For open subsets $U$ and $V$ of $Y$ we have $\cl U \cap \cl V = \cl(U \cap V)$. 
\end{lemma}

\begin{proof}
    Let $U$ and $V$ be open subsets of $Y$. Then there exist $U',V'\in\clopup(X)$ such that $U=U'\cap Y$ and $V=V'\cap Y$. Since $X$ is L-spatial, $Y$ is dense in $X$, so $U'=\cl U$ and $V'=\cl V$. Therefore, because $U'\cap V'$ is clopen in $X$, we have
    \[
    U'\cap V'=\cl((U'\cap V')\cap Y)=\cl((U'\cap Y)\cap (V'\cap Y))=\cl(U\cap V).
    \] 
    Thus, $\cl U \cap \cl V = U' \cap V' = \cl(U\cap V).$
\end{proof}

\begin{lemma} \label{lemma: prime filter cmap}
    Let $X_1, X_2$ be SL-spaces, $Y_1, Y_2$ their localic parts, $g : Y_1 \to Y_2$ a continuous map, and $x\in X_1$. Then  $\mathcal P_x := \{U \in \clopup(X_2) \mid x \in \cl[g^{-1}(U \cap Y_2)]\}$ is a prime filter in $\clopup(X_2)$. 
\end{lemma}

\begin{proof}
    It is easy to see that $\mathcal P_x$ is an upset and that $U \cup V \in \mathcal P_x$ implies $U \in \mathcal P_x$ or $V \in \mathcal P_x$. 
    Let $U,V \in \mathcal P_x$. Then $x\in \cl[g^{-1}(U\cap Y_2)], \cl[g^{-1}(V\cap Y_2)]$. Since $U \cap Y_2, V \cap Y_2$ are open in $Y_2$ and $g$ is continuous,  $g^{-1}(U \cap Y_2),g^{-1}(V \cap Y_2)$ are open in $Y_1$. Therefore, by \cref{lem:cl-respects-meet-2},
    \begin{align*}
    x \in \cl[g^{-1}(U\cap Y_2)] \cap \cl[g^{-1}(V\cap Y_2)]    
    &= \cl[g^{-1}(U\cap Y_2) \cap g^{-1}(V\cap Y_2)]\\
    &= \cl[g^{-1}((U \cap V) \cap Y_2)].
    \end{align*}
Thus, $U \cap V \in \mathcal P_x$, and hence $\mathcal P_x$ is a prime filter.
\end{proof}

\begin{lemma}  \label{lemma: g extends}
    Suppose that $X_1, X_2$ are SL-spaces, $Y_1, Y_2$ are their localic parts, and $g : Y_1 \to Y_2$ is a continuous map. Then there is an L-morphism $f : X_1 \to X_2$ which extends $g$.
\end{lemma}

\begin{proof}
    Let $x \in X_1$. By \cref{lemma: prime filter cmap}, $\mathcal P_{x} = \{U \in \clopup(X_2) \mid x \in \cl[g^{-1}(U \cap Y_2)]\}$ is a prime filter of $\clopup(X)$. By \cref{lemma:properties-Priestley-4}, $\bigcap \mathcal P_{x} = \upset z$ for a unique $z \in X_2$. Define $f : X_1 \to X_2$ by $f(x) = z$ for each $x \in X_1$. It is clear that $f$ is a well-defined map. To see that $f$ extends $g$, suppose $y \in Y_1$. Then 
    \begin{align*}
    \upset g(y)
    &= \bigcap \{U \in \clopup(X_2) \mid g(y) \in U\} \\
    &= \bigcap \{U \in \clopup(X_2) \mid g(y) \in U \cap Y_2\} \\
    &= \bigcap \{U \in \clopup(X_2) \mid y \in g^{-1}(U \cap Y_2)\} \\
    &= \bigcap \{U \in \clopup(X_2) \mid y \in \cl[g^{-1}(U \cap Y_2)]\} 
    = \bigcap \mathcal P_y,
    \end{align*}
where the second to last equality follows from \cref{cor:y-pulls-closure-2}.
Thus, $f(y) = g(y)$ by definition of $f$. 
    
    To see that $f$ is continuous, suppose $U$ is a clopen upset of $X_2$. Then $U \cap Y_2$ is an open subset of $Y_2$. Since $g$ is continuous, $g^{-1}(U \cap Y_2)$ is an open subset of $Y_1$, and hence $\cl g^{-1}(U \cap Y_2)$ is a clopen upset of $X_1$ (because $X_1$ is L-spatial). But $\cl g^{-1}(U \cap Y_2) = f^{-1}(U)$ since by definition of $f$ we have
    \begin{align*}
        x \in \cl g^{-1}(U \cap Y_2) 
        &\iff U \in \mathcal P_x 
        \iff \bigcap \mathcal P_x \subseteq U 
        \iff  \upset f(x) \subseteq U \\
        &\iff f(x) \in U
        \iff x \in f^{-1}(U),
    \end{align*}
    where in the second equivalence we use that $U$ is clopen, hence compact.
    Thus, $f^{-1}(U)$ is a clopen upset of $X_1$. Since clopen downsets are complements of clopen upsets, we also obtain $f^{-1}(D)$ is a clopen downset for each clopen downset $D$ of $X_2$. Thus, $f$ is continuous since clopen upsets and clopen downsets form a subbasis of $X_2$ (see \cref{lemma:properties-Priestley-0}).
    
    To see that $f$ is order-preserving, since $\cl g^{-1}(U \cap Y_2) = f^{-1}(U)$ is an upset, $x \leq z$ implies $\mathcal P_{x} \subseteq \mathcal P_z$. Therefore, $\bigcap \mathcal P_z \subseteq \bigcap \mathcal P_x$, and hence $f(x) \leq f(z)$. Thus, $f$ is order-preserving.

    It is left to prove that $\cl f^{-1}(U) = f^{-1} \cl U$ for each open upset $U$ of $X_2$. The left-to-right inclusion follows from the continuity of $f$. For the right-to-left inclusion, 
    let $x \in f^{-1}(\cl U)$. Then $f(x) \in \cl U$, so $\upset f(x) \subseteq \cl U$ by \cref{lem:esakia-int-downsets}. Therefore, $\bigcap \mathcal P_x \subseteq \cl U$ and $\cl U$ is open since $U$ is an open upset and $X_2$ is an L-space. Since $\mathcal P_x$ is a filter, by compactness there is $V \in \mathcal P_x$ such that $V \subseteq \cl U$. The former means that $x \in \cl g^{-1}(V \cap Y_2) = \cl f^{-1}(V \cap Y_2)$, which together with the latter and \cref{cor:y-pulls-closure-1} gives  
    \[
    x \in \cl f^{-1}(V \cap Y_2) \subseteq \cl    (f^{-1}(\cl U \cap Y_2)) = \cl f^{-1}(U \cap Y_2) \subseteq \cl f^{-1}(U).
    \]
    Thus, $f$ is an L-morphism.
\end{proof}

\begin{theorem} \label{thm: Y is full and faithful}
    $\functor Y$ is full and faithful.
\end{theorem}

\begin{proof}
    To see that $\functor Y$ is full, suppose $g : Y_1 \to Y_2$ is a  continuous map. By \cref{lemma: g extends}, there is an L-morphism $f : X_1 \to X_2$ extending $g$. Thus, $\functor Yf = g$.
    To see that $\functor Y$ is faithful, suppose $f_1, f_2 : X_1 \to X_2$ are L-morphisms with $f_1 \neq f_2$. Since $Y_1$ is a dense subset of $X_1$ and $X_2$ is Hausdorff, $f_1$ and $f_2$ must be the unique extensions of their restrictions $\functor Yf_1$ and $\functor Yf_2$ to $Y_1$ (see, e.g., \cite[p.~ 70]{Engelking1989}). Thus, $\functor Yf_1 \neq \functor Yf_2$. 
\end{proof}

\begin{corollary} \label{cor: sl=sob}
    \SL is equivalent to \Sob.
\end{corollary}

\begin{proof}
    By \cref{thm: Y is functor,thm: Y is full and faithful}, $\functor Y$ is essentially surjective, full, and faithful. Thus, $\functor Y$ is an equivalence (see, e.g., \cite[p.~93]{MacLane1998}).
\end{proof}

Combining \cref{thm: sfrm=sl,cor: sl=sob} yields an alternative proof of the well-known result mentioned in the introduction that \SFrm is dually equivalent to \Sob.
In the next section we will restrict the correspondence between \SFrm, \SL, and \Sob to obtain an alternative proof of Hofmann-Lawson duality.

\section{Deriving Hofmann-Lawson duality} 
\label{sec: 4}

\begin{definition} \label{def: kernel}
Suppose $X$ is an L-space. 
\begin{enumerate}
\item For $U,V\in\clopup(X)$, define $V \ll U$ provided for each open upset $W$ of $X$ we have $U \subseteq \cl W$ implies $V \subseteq W$.
\item For $U \in \clopup(X)$, define the \emph{kernel of $U$} as $$\ker U = \bigcup\{V \in \clopup(X) \mid V \ll U\}.$$ 
\end{enumerate}

If $X$ is the Priestley dual of a frame $L$ and $U = \varphi(a)$ for some $a \in L$, we simply write $\ker(a)$ for $\ker U$.
\end{definition}

\begin{lemma} \label{lem: ker properties}
    Let $X$ be an L-space and $U,V$ clopen upsets of $X$. 
    \begin{enumerate}[ref=\thelemma(\arabic*)]
        \item $\ker U$ is an open upset contained in $U$. \label[lemma]{lemma:kernel open upset}
        \item $\ker$ is monotone. \label[lemma]{lemma:kernel-monotone}
        \item $V \subseteq \ker U$ iff $V \ll U$. \label[lemma]{lemma:ll-iffs}
        \item $U \subseteq \cl W$ implies $\ker U \subseteq W$ for each open upset $W$. \label[lemma]{lemma:kernel-adjoint-of-closure}
    \end{enumerate}
    Moreover, if $X$ is the Priestley space of a frame $L$ and $a, b \in L$, then \begin{enumerate}[resume*]
    \item $a \ll b$ iff $\varphi(a) \ll \varphi(b)$ iff $\varphi(a) \subseteq \ker(b)$. \label[lemma]{lemma:ll-iffs-5}
\end{enumerate}
If in addition $L$ is spatial, then \begin{enumerate}[resume*]
    \item $a \ll b$ iff $\varphi(a) \subseteq \upset(\varphi(b) \cap Y_L)$.
    \end{enumerate}
\end{lemma}

\begin{proof}
    (1) This is immediate from the definition of $\ker U$ since $V \ll U$ implies $V \subseteq U$. 
        
    (2) Let $U_1 \subseteq U_2$, and let $V$ be a clopen upset with $V \ll U_1$. 
    Suppose $W$ is an open upset such that $U_2 \subseteq \cl W$. 
    Then $U_1 \subseteq \cl W$, so $V \subseteq W$. 
    Hence, $V \ll U_2$. Consequently, $\ker U_1 \subseteq \ker U_2$.
    
    (3) The right-to-left implication is immediate from the definition. For the left-to-right implication, if $V \subseteq \ker U$ then by compactness there is a clopen upset $V' \ll U$  such that $V \subseteq V'$. Therefore, $V \ll U$.
    
    (4) Suppose $U \subseteq \cl W$ and let $x \in \ker U$. Then there is a clopen upset $V$ of $X$ with $x \in V \ll U$. Hence, $x \in V \subseteq W$.
    
    (5) Suppose that $a \ll b$ and $U$ is an open upset of $X$ such that $\varphi(b) \subseteq \cl U$.
    Since $U = \bigcup \varphi[S]$ for some $S \subseteq L$, by \cref{lemma:joins-in-priestley}, we have
    $$\varphi(b) \subseteq \cl\left(\bigcup\varphi[S]\right) = \varphi\left(\bigvee S\right).$$ Therefore, $b \leq \bigvee S$.
    Since $a \ll b$, there is a finite $T\subseteq S$ such that $a \leq \bigvee T$. Thus,
    $$\varphi(a) \subseteq \varphi\left(\bigvee T\right) = \bigcup\varphi[T] \subseteq \bigcup \varphi[S] = U.$$
    Consequently, $\varphi(a) \ll \varphi(b)$.

    Conversely, suppose that $\varphi(a) \ll \varphi(b)$. Therefore, $\varphi(b) \subseteq \cl U$ implies $\varphi(a) \subseteq U$ for all open upsets $U \subseteq X_L$.
    Let $b \leq \bigvee S$ for some $S \subseteq L$.
    Then $$\varphi(b) \subseteq \varphi
    \left(\bigvee S\right) = \cl\bigcup \varphi[S].$$
    By assumption, $\varphi(a) \subseteq \bigcup \varphi[S]$. Since $\varphi(a)$ is compact, $\varphi(a) \subseteq \varphi[T] = \varphi(\bigvee T)$ for some finite $T\subseteq S$.
    Thus, $a \leq \bigvee T$, and hence $a \ll b$.
    
    This proves that $a \ll b$ iff $\varphi(a) \ll \varphi(b)$. The latter is equivalent to $\varphi(a) \subseteq \ker(b)$ by (3).
    
    (6) Suppose that $\varphi(a) \not\subseteq \upset(\varphi(b) \cap Y_L)$. Then there is $x \in \varphi(a)$ such that $\downset x \cap \varphi(b) \cap Y_L = \varnothing$. Therefore, $\varphi(b) \cap Y_L \subseteq (\downset x)^c$. Since $L$ is spatial, \cref{remark:spatial} implies that $\varphi(b) \subseteq \cl (\downset x)^c$. Hence, $\varphi(b)$ is contained in the closure of the open upset $(\downset x)^c$, while $\varphi(a) \not\subseteq (\downset x)^c$. Thus, $a \not\ll b$ by (5).
    
    For the converse, suppose that $\varphi(a) \subseteq \upset(\varphi(b) \cap Y_L)$ and $b \leq \bigvee S$ for some $S \subseteq L$. Then $\varphi(b) \subseteq \cl{\bigcup \varphi[S]}$. Therefore, $\varphi(b) \cap Y_L \subseteq \cl (\bigcup \varphi[S]) \cap Y_L = \bigcup \varphi[S] \cap Y_L \subseteq \bigcup\varphi[S]$ by \cref{cor:y-pulls-closure-1}. Thus, $\upset(\varphi(b) \cap Y_L) \subseteq \bigcup \varphi[S]$. By assumption, $\varphi(a) \subseteq \bigcup \varphi[S]$. Since $\varphi(a)$ is compact, $\varphi(a) \subseteq \bigcup \varphi[T]$ for some finite $T\subseteq S$. Hence, $a\le\bigvee T$, and so $a\ll b$.
\end{proof}

\begin{remark}
    The equivalence of the first two items of \cref{lemma:ll-iffs-5} was first proved in \cite[Prop.~3.6]{PultrSichler1988}.
\end{remark}

\begin{definition} \label{def: packed}
Let $X$ be an L-space.
\begin{enumerate}
\item We call a clopen upset $U$ of $X$ \emph{packed} if $\ker U$ is dense in $U$.
\item We call $X$ a \emph{continuous L-space} or simply a \emph{CL-space} if each clopen upset of $X$ is packed.
\end{enumerate}
\end{definition}

\begin{theorem}  
Let $L$ be a frame, $X_L$ its Priestley space, and $a\in L$.
    \begin{enumerate}[ref=\thetheorem(\arabic*)]
        \item $a = \bigvee \{b \in L \mid b \ll a\}$ iff $\varphi(a)$ is packed. \label[theorem]{thm:packed-iff-continuous}
        \item $L$ is a continuous frame iff $X_L$ is a CL-space.
                \label[theorem]{thm:hp-iff-continuous}
    \end{enumerate}
\end{theorem}

\begin{proof}
    (1) By \cref{lemma:joins-in-priestley,lemma:ll-iffs-5},
    \begin{align*}
        a = \bigvee \{b \in L \mid b \ll a\} 
        \iff \varphi(a) = \cl\ker(a) 
        \iff \ker(a) \text{ is dense in } \varphi(a).
    \end{align*}
    (2) This follows from (1).
\end{proof}

It is a well-known fact (see, e.g., \cite[p.~289]{Johnstone1982}) that the way below relation on a continuous frame $L$ is interpolating (meaning that $a\ll b$ implies $a\ll c\ll b$ for some $c\in L$).
In \cite[Lem.~5.3]{PultrSichler2000} an alternate proof of this result is given in the language of Priestley spaces:

\begin{lemma} \label{lemma:a-ll-b-implies-c-inbetween}
    Let $X$ be a CL-space and $U,V\in\clopup(X)$. If $U \ll V$, then there is $W\in\clopup(X)$ such that $U \ll  W \ll V$.
\end{lemma}

The next lemma is established in \cite[Sec.~5]{PultrSichler2000} (using different terminology).

\begin{lemma}
    \label{lemma:ll-iff-Scott-inbetween}
    Let $X$ be an L-space and $U, V\in\clopup(X)$. 
    \begin{enumerate}[ref=\thelemma(\arabic*)]
        \item If there is a Scott upset $F$ with $U \subseteq F \subseteq V$, then $U \ll V$. \label[lemma]{lemma-ll-iff-Scott-inbetween-1}
        \item If $X$ is a CL-space, then the converse of {\em (1)} also holds. \label[lemma]{lemma-ll-iff-Scott-inbetween-2}
    \end{enumerate}
\end{lemma}

It is well known (see, e.g., \cite[p.~311]{Johnstone1982}) that a continuous frame is spatial. In \cite[Prop.~4.6]{PultrSichler2000} an alternate proof of this result is given in the language of Priestley spaces:

\begin{theorem} \label{thm:hp-implies-spatial}
    If $X$ is a CL-space, then $X$ is L-spatial.
\end{theorem}

Consequently, if $X$ is a CL-space, then $Y$ is dense in $X$. We next prove that an L-spatial $X$ is a CL-space iff $Y$ is locally compact. For this we require the following lemma related to the Hofmann-Mislove Theorem \cite{HofmannMislove1981} (see also \cite[Thm.~II-1.20]{Compendium2003}). Recall that the Hofmann-Mislove Theorem establishes a (dual) isomorphism between the poset of compact saturated sets of a sober space $Y$ and the poset of Scott-open filters of the frame of opens of $Y$. Since Scott upsets correspond to Scott-open filters (see \cite[Lem~5.1]{BezhanishviliMelzer2022}), the lemma is in fact a version of the Hofmann-Mislove Theorem in the language of Priestley spaces.

\begin{lemma}[{\cite[Thm.~5.7]{BezhanishviliMelzer2022}}] \label{thm:sfilt-ksat}
    Let $L$ be a frame, $X_L$ its Priestley space, and $Y_L \subseteq X_L$ the localic part of $X_L$. The map $F \mapsto F \cap Y_L$ is an isomorphism from the poset of Scott upsets of $X_L$ to the poset of compact saturated sets of $Y_L$ $($both ordered by inclusion$)$. The inverse isomorphism is given by $K \mapsto \upset{K}$.
    \end{lemma}

\begin{theorem} \label{thm:hp-iff-lc}
Let $L$ be a spatial frame, $X_L$ its Priestley space, and $Y_L\subseteq X_L$ the localic part of $X_L$. Then $X_L$ is a CL-space iff $Y_L$ is locally compact.
\end{theorem}

\begin{proof}
 First suppose that $X_L$ is a CL-space, $y\in Y_L$, and $\zeta(a)$ is an open neighborhood of $y$. Since $\zeta(a) = \varphi(a) \cap Y_L$ (see \cref{lemma:zeta-is-phi-restricted}), we have 
\[y \in \varphi(a) = \cl \ker(a) = \cl \bigcup\{\varphi(b) \mid \varphi(b) \ll \varphi(a)\}.\] Because $\downset y$ is open, $y \in \varphi(b)$ for some $\varphi(b) \ll \varphi(a)$. 
Therefore, $y \in \varphi(b) \cap Y_L = \zeta(b)$. 
By Lemma~\ref{lemma-ll-iff-Scott-inbetween-2}, there is a Scott upset $F$ such that $\varphi(b)\subseteq F\subseteq\varphi(a)$. Thus, \[y\in\zeta(b)\subseteq F\cap Y_L\subseteq\zeta(a).\] By \cref{thm:sfilt-ksat}, $F \cap Y_L$ is compact. Consequently, $Y_L$ is locally compact.
    
    Conversely, suppose that $Y_L$ is locally compact and $a\in L$. We must show that $\ker(a)$ is dense in $\varphi(a)$. Let $x \in \varphi(a)$ and $W$ be an open neighborhood of $x$ in $X_L$. By \cref{lemma:properties-Priestley-0}, there exist clopen upsets $U$ and $V$ of $X_L$ such that $x\in U\cap V^c \subseteq W$. Therefore we have $U\cap V^c\cap\varphi(a)\ne\varnothing$. Because $L$ is spatial, $Y_L$ is dense in $X_L$, so $U \cap V^c \cap \zeta(a) \neq \varnothing$, and hence there is $y \in U\cap V^c\cap \zeta(a)$. Since $Y_L$ is locally compact, there is $b\in L$ and a compact saturated $K \subseteq Y_L$ such that $y\in\zeta(b) \subseteq K \subseteq \zeta(a)$. 
By \cref{thm:sfilt-ksat}, $\upset K$ is a Scott upset. Thus, $\upset K$ is closed, and so $\varphi(b) = \cl \zeta(b) \subseteq \upset K$ by \cref{remark:spatial} (which is applicable since $X_L$ is L-spatial by \cref{thm: spatial}).
 Therefore, $\varphi(b) \subseteq \upset K \subseteq \varphi(a)$. Then $\varphi(b) \ll \varphi(a)$ by Lemma~\ref{lemma-ll-iff-Scott-inbetween-1}. Thus, $y \in \ker(a)$ by Lemma~\ref{lemma:ll-iffs}. This implies that $U\cap V^c\cap \ker(a) \neq \varnothing$, so $\ker(a)$ is dense in $\varphi(a)$. 
\end{proof}

\cref{thm:hp-iff-continuous,thm:hp-iff-lc} establish a one-to-one correspondence between continuous frames, CL-spaces, and locally compact sober spaces. Next, we extend these to categorical equivalences.

\begin{lemma} \label{lemma:proper-iff-priestley}
    Let $h : L_1 \to L_2$ be a frame homomorphism and $f : X_{L_2} \to X_{L_1}$ its dual L-morphism. 
    Then $h$ is proper iff \begin{equation*}
    f^{-1}(\ker U) \subseteq \ker f^{-1}(U) \tag{$\sharp$} \label{eq:properness}
\end{equation*}
for all $U\in\clopup(X_{L_1})$.
\end{lemma}

\begin{proof}
    First suppose that $h$ is proper and $U\in\clopup(X_{L_1})$. Let $x \in f^{-1}(\ker U)$. Then $f(x) \in \ker U$. Therefore, there exists $V\in\clopup(X_{L_1})$ with $f(x) \in V \ll U$. Since $U,V \in \clopup(X_{L_1})$, there exist $a,b \in L_1$ with $\varphi(a) = V$ and $\varphi(b) = U$. Then $a \ll b$ by \cref{lemma:ll-iffs-5}. Since $h$ is proper, $ha \ll hb$, and hence using \cref{lemma:ll-iffs-5} again, $f^{-1}(V) = \varphi(ha) \ll \varphi(hb) = f^{-1}(U)$.
    Thus, $x \in f^{-1}(V) \ll f^{-1}(U)$, and so $x \in \ker f^{-1}(U)$. 
    
    Conversely, suppose (\ref{eq:properness}) holds for all $U\in\clopup(X_{L_1})$. Let $a \ll b$. Then $\varphi(a) \subseteq \ker(b)$ by Lemma~\ref{lemma:ll-iffs-5}. Therefore, $f^{-1}(\varphi(a)) \subseteq f^{-1}(\ker(b))$. Thus, $f^{-1}(\varphi(a)) \subseteq \ker f^{-1}(\varphi(b))$ by (\ref{eq:properness}). Consequently, $f^{-1}(\varphi(a)) \ll f^{-1}(\varphi(b))$ by \cref{lemma:ll-iffs}. Hence, $\varphi(ha) \ll \varphi(hb)$, and so $ha \ll hb$ by \cref{lemma:ll-iffs-5}, yielding that $h$ is proper. 
\end{proof}

\begin{definition} \label{def: proper L-morphism}
    Let $f : X_1 \to X_2$ be an L-morphism between L-spaces. We call $f$ \emph{proper} if $f$ satisfies (\ref{eq:properness}) for all clopen upsets of $X_2$.
\end{definition}

It is straightforward to check that CL-spaces and proper L-morphisms form a category, which we denote 
by {\CL}.

\begin{theorem} \label{thm:hl-frames}
    {\CFrm} is dually equivalent to {\CL}.
\end{theorem}

\begin{proof}
    The units $\varphi : L \to \functor D\functor X(L)$ and $\varepsilon : X \to \functor X\functor D(X)$ of Pultr-Sichler duality (see \cref{rem: units}) remain isomorphisms in \CFrm and \CL. 
    Thus, it follows from \cref{thm:hp-iff-continuous,lemma:proper-iff-priestley} that the restrictions of the functors $\functor X$ and $\functor D$ yield a dual equivalence between {\CFrm} and {\CL}.
\end{proof}

We next give several equivalent conditions for an L-morphism between CL-spaces to be proper. For this we require the following lemma, item (1) of which generalizes \cite[Lem.~4.5]{PultrSichler2000} and provides means to find Scott upsets. 

\begin{lemma} 
    Let $X$ be a CL-space.
    \begin{enumerate}[ref=\thelemma(\arabic*)]
    \item If $\mathcal U \subseteq \clopup(X)$  is a down-directed family  such that $\bigcap \mathcal U = \bigcap\{\ker U \mid U \in \mathcal U\}$, then $\bigcap \mathcal U$ is a Scott upset.\label[lemma]{lemma:intersections-of-kernels-scott}
    \item $\ker U = \upset (U \cap Y)$ for every $U\in\clopup(X)$.\label[lemma]{lemma:hp-kernels}
\end{enumerate} 
\end{lemma}

\begin{proof}
(1) Clearly $\bigcap \mathcal U$ is a closed upset. To see that it is a Scott upset, by Lemma~\ref{lem: Scott upset} it is enough to show that $\bigcap \mathcal U \subseteq \cl V$ implies $\bigcap \mathcal U \subseteq V$ for every open upset $V$ of $X$. 
    Note that $\cl V$ is open since $X$ is an L-space. Therefore, since $X$ is compact and $\bigcap \mathcal U$ is down-directed, from $\bigcap \mathcal U \subseteq \cl V$ it follows that there is $U \in \mathcal U$ with $U \subseteq \cl V$. Thus, $\ker U \subseteq V$ by \cref{lemma:kernel-adjoint-of-closure}. Since $U \in \mathcal U$ and $\bigcap \mathcal U = \bigcap \{\ker U \mid U \in \mathcal U\}$, we have $\bigcap \mathcal U \subseteq \ker U$. 
    Consequently, $\bigcap \mathcal U  \subseteq V$.

(2) First suppose that $x \in \ker U$. Then there is $V\in\clopup(X)$ with $x \in V \ll U$. By \cref{lemma-ll-iff-Scott-inbetween-2}, there is a Scott upset $F$ with $V \subseteq F \subseteq U$. Therefore, there is $y \in F \cap Y$ with $y \leq x$. Thus, $x \in \upset(U \cap Y)$.

Conversely, suppose that $x \in \upset(U \cap Y)$. Then there is $y \in U \cap Y$ with $y \leq x$. Since $U$ is packed, $U=\cl\ker U$, so $\upset y \subseteq \cl\ker U$. Thus, since $\upset y$ is a Scott upset and $\ker U$ is an open upset, $x \in \upset y \subseteq \ker U$ by Lemma~\ref{lem: Scott upset}. 
\end{proof}

\begin{theorem} \label{thm: equiv conditions}
    Let $X_1$ and $X_2$ be CL-spaces, $Y_1$ and $Y_2$ the localic parts of $X_1$ and $X_2$ respectively, and $f : X_1 \to X_2$ an L-morphism. The following are equivalent.
    \begin{enumerate}[ref=\thetheorem(\arabic*)]
        \item $f$ is proper.
        \item $f^{-1}\upset(U \cap Y_2) = \upset(f^{-1}(U) \cap Y_1)$ for all $U\in\clopup(X_2)$.
        \item $f^{-1}(\upset y)$ is a Scott upset of $X_1$ for all $y \in Y_2$. \label[theorem]{thm: equiv conditions-3}
        \item $f^{-1}(F)$ is a Scott upset of $X_1$ for all Scott upsets $F$ of $X_2$. \label[theorem]{thm: equiv conditions-4}
        \item $\downset f(x) \cap Y_2\subseteq \downset f(\downset x \cap Y_1)$ for all $x \in X_1$.
    \end{enumerate}
\end{theorem}

\begin{proof}
(1)$\Rightarrow$(2) Suppose $x \in f^{-1}\upset(U \cap Y_2)$. Then $x \in f^{-1}(\ker U)$ by \cref{lemma:hp-kernels}. Since $f$ is proper, $x \in \ker f^{-1}(U)$, and using \cref{lemma:hp-kernels} again yields $x \in \upset(f^{-1}(U) \cap Y_1)$; for the converse, suppose $x \in \upset(f^{-1}(U) \cap Y_1)$. Then $x \geq y$ for some $y \in f^{-1}(U) \cap Y_1$. Therefore, $f(x) \geq f(y)$ and $f(y) \in U$. By Lemma~\ref{lem: f restricts}, $f(Y_1) \subseteq Y_2$. Thus, $f(y) \in U \cap Y_2$, so $f(x) \in \upset(Y \cap Y_2)$, and hence $x\in f^{-1}\upset(Y \cap Y_2)$. 

(2)$\Rightarrow$(3) Since $\upset y$ is a closed upset, $\upset y = \bigcap\{U \in \clopup(X_2) \mid y \in U\}$ by \cref{lemma:properties-Priestley-2}. Therefore, since $y \in Y_2$, we have $\upset y = \bigcap \{ \upset (U \cap Y_2) \mid y \in U\in \clopup(X_2) \}$. Thus, by (2) and Lemma~\ref{lemma:hp-kernels},
\begin{align*}
\bigcap \{ f^{-1}(U) \mid y \in U\in \clopup(X_2) \} 
&= f^{-1}\left( \bigcap \{U \in \clopup(X_2) \mid y \in U\} \right) \\
&= f^{-1}\left(\bigcap \{ \upset (U \cap Y_2) \mid y \in U\in \clopup(X_2) \} \right) \\
&= \bigcap \{ f^{-1}\upset (U \cap Y_2) \mid y \in U\in \clopup(X_2) \} \\
&= \bigcap \{ \upset(f^{-1}(U) \cap Y_1) \mid y \in U\in \clopup(X_2) \} \\
&= \bigcap \{ \ker f^{-1}(U) \mid y \in U\in \clopup(X_2) \}.
\end{align*}
Consequently, 
\[
f^{-1}(\upset y) = \bigcap \{ f^{-1}(U) \mid y \in U\in \clopup(X_2) \} = \bigcap \{ \ker f^{-1}(U) \mid y \in U\in \clopup(X_2) \}
\] 
is a Scott upset by \cref{lemma:intersections-of-kernels-scott}.

(3)$\Rightarrow$(4) Let $F$ be a Scott upset of $X_2$. By (3) we have
    \begin{align*}
        \min f^{-1}(F) 
        &= \min f^{-1}\bigcup \{\upset y \mid y \in \min F\} \\
        &= \min \bigcup \{f^{-1}(\upset y) \mid y \in \min F\} \\
        &\subseteq \bigcup \{\min f^{-1}(\upset y) \mid y \in \min F\} \subseteq Y_1.
    \end{align*}
    Thus, $f^{-1}(F)$ is a Scott upset of $X_1$.

(4)$\Rightarrow$(5) Suppose $y_2 \in \downset f(x) \cap Y_2$. Then $\upset y_2$ is a Scott upset of $X_2$, so $f^{-1}(\upset y_2)$ is a Scott upset of $X_1$ by (4).
Since $x \in f^{-1}(\upset y_2)$, there is $y_1 \in \min f^{-1}(\upset y_2)$ such that $y_1 \leq x$. Therefore, $y_2 \leq f(y_1)$ and $y_1 \in \downset x \cap Y_1$. Thus, $y_2 \in \downset f(\downset x \cap Y_1)$.

(5)$\Rightarrow$(1) Let $x \in f^{-1}(\ker U)$. Then $f(x) \in \ker(U)$, and \cref{lemma:hp-kernels} implies that $f(x) \in \upset(U \cap Y_2)$. Therefore, there is $y \in \downset f(x) \cap (U \cap Y_2)$. By (5), $y \in \downset f(\downset x \cap Y_1)$, so there is $y' \in \downset x \cap Y_1$ with $y \leq f(y')$. Thus, $f(y') \in U$, and hence $y' \in f^{-1}(U) \cap Y_1$. Consequently, \cref{lemma:hp-kernels} yields that $x \in \upset (f^{-1}(U) \cap Y_1) = \ker(f^{-1}(U) \cap Y_1) \subseteq \ker f^{-1}(U)$. 
\end{proof}

Let $h : L_1 \to L_2$ be a frame homomorphism between continuous frames, $f : X_{L_2} \to X_{L_1}$ its dual L-morphism, and $\functor Y f: Y_{L_2} \to Y_{L_1}$ the restriction of $f$. 
The next theorem characterizes when each of these maps is proper.

\begin{theorem} \label{thm:proper-maps-eqv}
     The following are equivalent.     
    \begin{enumerate}
        \item $h : L_1 \to L_2$ is a proper frame homomorphism.
        \item $f : X_{L_2} \to X_{L_1}$ is a proper L-morphism.
        \item $\functor Y f : Y_{L_2} \to Y_{L_1}$ is a proper map.    
    \end{enumerate}
\end{theorem}

\begin{proof}
(1)$\Leftrightarrow$(2) This follows from \cref{lemma:proper-iff-priestley}.
    
(2)$\Rightarrow$(3)
We let $g = \functor Y f$ and verify that $g$ satisfies \cref{def:2.4(2)}. By \cref{rem:2.5}, it is sufficient to show that $g^{-1}(U)$ is compact for each compact saturated $U$ in $Y_{L_1}$.
Since $U$ is compact saturated in $Y_{L_1}$, we have that $\upset U$ is a Scott upset of $X_{L_1}$ by 
\cref{thm:sfilt-ksat}.
Hence, $f^{-1}(\upset U)$ is a Scott upset of $X_{L_2}$ by \cref{thm: equiv conditions-4}. Thus, $f^{-1}(\upset U) \cap Y_{L_2}$ is compact saturated in $Y_{L_2}$ by
\cref{thm:sfilt-ksat}. But $f^{-1}(\upset U) \cap Y_{L_2} = g^{-1}(U)$ because $U$ is saturated in $Y_{L_2}$ and $g$ is the restriction of $f$ to $Y_{L_2}$. Therefore, $g^{-1}(U)$ is compact.

(3)$\Rightarrow$(2) By \cref{thm: equiv conditions-3}, it is enough to show that $f^{-1}(\upset y)$ is a Scott upset of $X_{L_2}$ for each $y \in Y_{L_1}$. Since $y \in Y_{L_1}$, we have that $\upset y$ is a Scott upset of $X_{L_1}$, so $\upset y \cap Y_{L_1}$ is compact saturated in $Y_{L_1}$ by \cref{thm:sfilt-ksat}. Let $g = \functor Y f$. Because $g$ is proper, $g^{-1}(\upset y \cap Y_{L_1})$ is compact saturated in $Y_{L_2}$. Hence, $\upset g^{-1}(\upset y \cap Y_{L_1})$ is a Scott upset of $X_{L_2}$ by \cref{thm:sfilt-ksat}.  Therefore, it suffices to show that $f^{-1}(\upset y) = \upset g^{-1}(\upset y \cap Y_{L_1})$.

Clearly $\upset g^{-1}(\upset y \cap Y_{L_1}) \subseteq f^{-1}(\upset y)$. For the reverse inclusion, suppose $x \notin \upset g^{-1}(\upset y \cap Y_{L_1})$. Then there is a clopen downset $D$ of $X_{L_2}$ such that $x \in D$ and $D \cap g^{-1}(\upset y \cap Y_{L_1}) = \varnothing$.
Hence, $y \notin \downset g(D \cap Y_{L_2})$. Since $g$ is proper, $\downset g(D \cap Y_{L_2}) \cap Y_{L_1}$ is closed in $Y_{L_1}$, and so $\downset g(D \cap Y_{L_2}) \cap Y_{L_1} = E \cap Y_{L_1}$ for some clopen downset $E$ of $X_{L_1}$. Therefore, $y \not \in E$ and $g(D \cap Y_{L_2}) \subseteq E$, so $\downset \cl g(D \cap Y_{L_2}) \subseteq E$. Because $X_{L_1}$ and $X_{L_2}$ are CL-spaces, they are L-spatial by \cref{thm:hp-implies-spatial}. Thus, since $f$ is a closed map, we have
\[
\downset f(D) = \downset f \cl(D \cap Y_{L_2}) = \downset \cl f(D \cap Y_{L_2}) = \downset \cl g(D \cap Y_{L_2}) \subseteq E.
\] Consequently, $y \not \in \downset f(D)$, and hence $x \not \in f^{-1}(\upset y)$.
\end{proof}

\begin{corollary}  \label{cor: proper}
    Suppose $X_1$, $X_2$ are CL-spaces and $g : Y_1 \to Y_2$ is a proper map between their localic parts. Then there is a proper L-morphism $f : X_1 \to X_2$ extending $g$. \label{lemma: proper extension}
\end{corollary}

\begin{proof}
    By \cref{lemma: g extends}, there is an L-morphism $f: X_1 \to X_2$ extending $g$. Thus, $\functor Y f = g$, and so $f$ is proper by \cref{thm:proper-maps-eqv}.
\end{proof}

\begin{theorem}\label{thm:hl-spaces}
    {\CL} is equivalent to {\LCSob}.
\end{theorem}

\begin{proof} 
It follows from \cref{thm:hp-iff-lc} that the restriction of $\functor Y$ is well defined on objects. By \cref{thm:proper-maps-eqv},
the restriction of $\functor Y$ is also well defined on morphisms. This together with \cref{cor: proper} shows that \cref{thm: Y is functor,thm: Y is full and faithful} apply to yield that the restriction $\functor Y : {\CL} \to {\LCSob}$ is essentially surjective, full, and faithful.
\end{proof}

We thus obtain the following alternative proof of Hofmann-Lawson duality:
  
\begin{corollary}[Hofmann-Lawson]
    {\CFrm} is dually equivalent to {\LCSob}.
\end{corollary}

\begin{proof}
   Combine \cref{thm:hl-frames,thm:hl-spaces}.
\end{proof}

\section{Deriving dualities for stably continuous frames} \label{sec: 5}

To derive the two dualities for stably continuous frames, we first characterize stability of $\ll$ in the language of Priestley spaces.

\begin{lemma} \label{lemma:locally-stably-compact-iff-kernels-intersections}
Let $L$ be a continuous frame, $X_L$ its Priestley space, and $Y_L \subseteq X_L$ the localic part of $X_L$. For $a, b \in L$ we have 
\[
(\forall c\in L)(c \ll a,b \Rightarrow c \ll a \wedge b) 
\mbox{ iff } \ker(a) \cap \ker(b) = \ker(a \wedge b).
\]
\end{lemma}

\begin{proof}
    First suppose that $(\forall c\in L)(c \ll a,b \Rightarrow c \ll a \wedge b)$. Then \cref{lemma:ll-iffs-5} gives
\[
   (\forall c\in L)( \varphi(c) \ll \varphi(a),\varphi(b) \implies \varphi(c) \ll \varphi(a) \cap \varphi(b) ). \tag{$\star$} \label{eq:star}
\]   
 Since $\ker$ is monotone (see \cref{lemma:kernel-monotone}), $\ker(a \wedge b) \subseteq \ker(a) \cap \ker(b)$. For the reverse inclusion, let $x \in \ker(a) \cap \ker(b)$. Then there are $d,e\in L$ with $x \in \varphi(d) \ll \varphi(a)$ and $x \in \varphi(e) \ll \varphi(b)$. Let $c=d\wedge e$. Then $x \in \varphi(c) \ll \varphi(a), \varphi(b)$. Consequently, by (\ref{eq:star}), $\varphi(c) \ll \varphi(a) \cap \varphi(b) = \varphi(a \wedge b)$. 
    Therefore, $x \in \ker(a \wedge b)$ by \cref{lemma:ll-iffs-5}. Thus, $\ker(a) \cap \ker(b) = \ker(a \wedge b)$. 
    
    For the converse, suppose that $\ker(a) \cap \ker(b) = \ker(a \wedge b)$. Let $c \in L$ with $c \ll a,b$. Then $\varphi(c) \subseteq \ker(a),\ker(b)$ by \cref{lemma:ll-iffs-5}. Therefore, $\varphi(c) \subseteq \ker(a) \cap \ker(b) = \ker(a \wedge b)$. Thus, using \cref{lemma:ll-iffs-5} again, we obtain $c\ll a \wedge b$ .
\end{proof}

The previous lemma motivates defining stability in terms of kernels commuting with intersections. We will see in \cref{lemma:kernel-stable-Scott} that for CL-spaces this property coincides with the property that binary intersections of Scott upsets are Scott upsets.

\begin{definition} \label{def: scott-stable}
    Let $X$ be an L-space. 
    \begin{enumerate}
    \item We call $X$ \emph{kernel-stable} if for all clopen upsets $U$ and $V$ of $X$ we have 
    \[\ker U \cap \ker V = \ker(U \cap V).\] 
    \item We call $X$ \emph{Scott-stable} if for all Scott upsets $F$ and $G$ of $X$ we have that $F \cap G$ is a Scott upset.
    \end{enumerate}
\end{definition}

\begin{lemma} \label{lemma:kernel-stable-Scott}
    Let $X$ be a CL-space.
    \begin{enumerate}[ref=\thelemma(\arabic*)]
        \item For every Scott upset $F$, we have $F = \bigcap\{\ker U \mid F \subseteq U \in \clopup(X) \}$.
        \item $X$ is kernel-stable iff $X$ is Scott-stable. \label[lemma]{lemma:kernel-stable-Scott-2}
    \end{enumerate}
\end{lemma}
\begin{proof}
    (1) Suppose $F \subseteq U \in \clopup(X)$. Since $X$ is a CL-space, $F \subseteq \cl\ker U$. By \cref{lemma:kernel open upset}, $\ker U$ is an open upset. Therefore, $F \subseteq \ker U$ by \cref{lem: Scott upset}. 
    Thus, $F \subseteq \bigcap \{\ker U \mid F \subseteq U \in \clopup(X) \}$. For the reverse inclusion, by \cref{lemma:properties-Priestley-2}, we have
    \[F = \bigcap\{U \mid F \subseteq U \in \clopup(X) \} \supseteq \bigcap \{\ker U \mid F \subseteq U \in \clopup(X) \}.\]
    
    (2) Suppose $X$ is kernel-stable. Let $F,G$ be Scott upsets. If $U,V,W$ range over clopen upsets of $X$, by (1) we have 
    \begin{align*}
        F \cap G 
        &= \bigcap\{\ker U \mid F \subseteq U\} \cap \bigcap \{\ker V \mid G \subseteq V\} \\
        &= \bigcap\{\ker U \cap \ker V \mid F \subseteq U, G \subseteq V\} \\
        &= \bigcap\{\ker (U \cap V) \mid F \subseteq U, G \subseteq V\}\\
        &= \bigcap\{\ker W \mid F \cap G \subseteq W\} \\
        &\subseteq \bigcap\{W \mid F \cap G \subseteq W\} 
        =  F \cap G,
    \end{align*}
    where the last equality follows from \cref{lemma:properties-Priestley-2}. For the second to last equality it is enough to observe that by compactness, $F \cap G \subseteq W$ is equivalent to $U \cap V \subseteq W$ for some clopen upsets $U \supseteq F$ and $V \supseteq G$.
    Thus, $F \cap G$ is a Scott upset by \cref{lemma:intersections-of-kernels-scott}, and hence $X$ is Scott-stable. 
    
    Conversely, suppose $X$ is Scott-stable. Let $U,V$ be clopen upsets. Since $\ker U$ is an open upset for each $U$ (see \cref{lemma:kernel open upset}), it suffices to show that $W \subseteq \ker(U) \cap \ker (V)$ iff $W \subseteq \ker(U \cap V)$ for each clopen upset $W$ (see \cref{lemma:properties-Priestley-1}). Let $W$ be a clopen upset. 
By \cref{lemma:ll-iffs}, $W \subseteq \ker(U) \cap \ker(V)$ iff $W \ll U,V$. By \cref{lemma:ll-iff-Scott-inbetween}, this happens iff there are Scott upsets $F$ and $G$ such that $W \subseteq F \subseteq U$ and $W \subseteq G \subseteq V$. Since $X$ is Scott-stable, the latter is equivalent to the existence of a Scott upset $H$ such that $W \subseteq H \subseteq U \cap V$. By invoking \cref{lemma:ll-iff-Scott-inbetween} again, this is equivalent to $W \ll U \cap V$, which in turn is equivalent to $W \subseteq \ker(U \cap V)$ by \cref{lemma:ll-iffs}. Thus, $X$ is kernel-stable.
\end{proof}

\begin{theorem} \label{thm:shp-iff-SCFrm}
    Let $L$ be a frame and $X_L$ its Priestley space. Then $L$ is a stably continuous frame iff $X_L$ is a Scott-stable CL-space.
\end{theorem}

\begin{proof}
Apply \cref{thm:hp-iff-continuous}, \cref{lemma:locally-stably-compact-iff-kernels-intersections}, and \cref{lemma:kernel-stable-Scott-2}.
\end{proof}

\begin{definition} \label{def: StCL}
We call an L-space {\em stably continuous} or simply a \emph{StCL-space} if it is a Scott-stable CL-space. Let {\SCL} be the full subcategory of {\CL} consisting of StCL-spaces.
\end{definition}

\begin{corollary} \label{thm: st cont}
{\SCFrm} is dually equivalent to {\SCL}.
\end{corollary}

\begin{proof}
Restrict \cref{thm:hl-frames} to the full subcategories \SCFrm and \SCL using \cref{thm:shp-iff-SCFrm}.
\end{proof}

Next we show that \SCL is equivalent to \SLCSp.

\begin{theorem} \label{thm:shp-iff-stlc}
Let $X$ be an SL-space and $Y$ its localic part. Then $X$ is a StCL-space iff $Y$ is stably locally compact.
\end{theorem}

\begin{proof}
Suppose $X$ is a StCL-space. By \cref{thm:hp-iff-lc}, $Y$ is locally compact. Let $K,J$ be compact saturated sets in $Y$. Then $\upset K$ and $\upset J$ are Scott upsets by \cref{thm:sfilt-ksat}. Since $X$ is Scott-stable, $\upset K \cap \upset J$ is a Scott upset. Therefore, $K\cap J = \upset K \cap \upset J \cap Y$ is compact saturated by \cref{thm:sfilt-ksat}.

Conversely, suppose that $Y$ is stably locally compact. By \cref{thm:hp-iff-lc}, $X$ is a CL-space. By \cref{lemma:kernel-stable-Scott-2}, it is enough to show that $X$ is kernel-stable. Let $U,V \in \clopup(X)$. Since $\ker$ is monotone by \cref{lemma:kernel-monotone}, it suffices to show that $\ker U \cap \ker V \subseteq \ker(U \cap V)$. Let $x \in \ker U \cap \ker V$. Then there exist $U',V' \in \clopup(X)$ containing $x$ such that $U' \ll U$ and $V' \ll V$. By \cref{lemma-ll-iff-Scott-inbetween-2}, there are Scott upsets $F$, $G$ with $U' \subseteq F \subseteq U$ and $V' \subseteq G \subseteq V$. By \cref{thm:sfilt-ksat}, $F \cap Y$ and $G \cap Y$ are compact saturated. Since $Y$ is stably locally compact, $F \cap G \cap Y$ is compact saturated. Hence, $\upset (F \cap G \cap Y)$ is a Scott upset by \cref{thm:sfilt-ksat}. Moreover, because $F \cap G \subseteq U \cap V$, we have $\upset (F \cap G \cap Y) \subseteq \ker (U \cap V)$ by \cref{lemma:hp-kernels}. Therefore, since $X$ is L-spatial, $x\in U' \cap V' = \cl (U' \cap V' \cap Y) \subseteq \upset(F \cap G \cap Y) \subseteq \ker (U \cap V)$.
\end{proof}

\begin{corollary} \label{thm: st loc comp}
{\SCL} is equivalent to {\SLCSp}.
 \end{corollary}

\begin{proof}
Restrict \cref{thm:hl-spaces} to the full subcategories {\SCL} and {\SLCSp} using \cref{thm:shp-iff-stlc}.
\end{proof}

As a consequence of \cref{thm: st cont,thm: st loc comp}, we obtain the following well-known duality for stably continuous frames (see \cref{SCFrm-deqv-SLCsp}):

\begin{corollary}
{\SCFrm} is dually equivalent to {\SLCSp}.
\end{corollary}

We next turn our attention to stably compact frames.

\begin{lemma} \label{lemma:compactness}
    Let $L$ be a frame and $X_L$ its Priestley space. For $a \in L$, the following are equivalent.
\begin{enumerate}
    \item $a$ is compact.
    \item $\ker(a) = \varphi(a)$.
    \item $\varphi(a)$ is a Scott upset.
    \end{enumerate}
    In particular, $L$ is compact iff $X_L = \ker X_L$ iff $X_L$ is a Scott upset.
\end{lemma}
\begin{proof}
    (1)$\Rightarrow$(2) 
    This follows from
    \namecref{lem: ker properties} \hyperref[lem: ker properties]{\labelcref{lem: ker properties}(1,5)}.
    
    (2)$\Rightarrow$(3) Suppose $\varphi(a) \subseteq \cl U$ for some open upset $U$ of $X_L$. Then $\varphi(a) = \ker(a) \subseteq U$ by \cref{lemma:kernel-adjoint-of-closure}. Therefore, $\varphi(a)$ is a Scott upset by \cref{lem: Scott upset}.
    
    (3)$\Rightarrow$(1) Since $\varphi(a)$ is a Scott upset, $\varphi(a) \ll \varphi(a)$ by \cref{lemma-ll-iff-Scott-inbetween-1}. Thus, $a \ll a$ by \cref{lemma:ll-iffs-5}. 
    
    The last statement follows from the above equivalence and the fact that $X_L=\varphi(1)$.
\end{proof}

\begin{remark}
    The equivalence (1)$\Leftrightarrow$(3) of \cref{lemma:compactness} is known (see, e.g., \cite[Cor.~5.4]{BezhanishviliMelzer2022}), and so is the fact that $L$ is compact iff $X_L$ is a Scott upset (see \cite[Thm.~3.5]{PultrSichler1988} or \cite[Lem.~3.1]{BezhanishviliGabelaiaJibladze2016}). To this \cref{lemma:compactness} adds a characterization in terms of kernels. 
\end{remark}

\begin{definition} \label{def: TPL}
Let $X$ be an L-space. 
\begin{enumerate}
    \item We call $X$ \emph{L-compact} if $X$ is a Scott upset.
    \item We call $X$ \emph{stably L-compact} or simply a \emph{StKL-space} if $X$ is an L-compact StCL-space.
    \item Let {\SKL} be the full subcategory of {\SCL} consisting of StKL-spaces.
\end{enumerate}
\end{definition}

As an immediate consequence of \cref{thm:shp-iff-SCFrm} and \cref{lemma:compactness}, we obtain:

\begin{theorem} \label{thm:shtp-iff-stk}
 Let $L$ be a frame and $X_L$ its Priestley space. Then $L$ is a stably compact frame iff $X_L$ is a StKL-space.
\end{theorem}
This together with \cref{thm: st cont} yields:

\begin{corollary} \label{thm: st comp = tpl}
{\SKFrm} is dually equivalent to {\SKL}.
\end{corollary}

To connect StKL-spaces with stably compact spaces, we need the following lemma.

\begin{lemma} \label{lemma:scott-iff-compact}
    Let $X$ be an L-space and $Y$ the localic part of $X$. If $X$ is L-compact, then $Y$ is compact. If in addition $X$ is L-spatial, then the converse also holds.
\end{lemma}

\begin{proof}
    Suppose $X$ is L-compact. Let $\mathcal U$ be an open cover of $Y$. For each $U' \in \mathcal U$ there is a clopen upset $U$ of $X$ such that $U \cap Y = U'$. Since $X$ is L-compact, \[\min(X) \subseteq Y \subseteq \bigcup\{ U \mid U' \in \mathcal U \}.\] Therefore, $X = \upset(\min X) \subseteq \bigcup\{ U \mid U' \in \mathcal U \}$ because the latter is an upset. Since $X$ is compact, there are $U_1',\dots,U_n' \in \mathcal U$ such that $X \subseteq U_1 \cup \dots \cup U_n$. But then we have $Y \subseteq (U_1 \cup \dots \cup U_n) \cap Y = U'_1 \cup \dots \cup U'_n$. Thus, $Y$ is compact. 

    Conversely, suppose $X$ is not L-compact, so there is $x \in \min(X) \setminus Y$. Then $\downset x = \{x\}$ is not open. Hence, $U = \{x\}^c$ is an open upset such that $x \in \cl U$. Since $U$ is an open upset, $U = \bigcup \{V \in \clopup(X) \mid x \not \in V\}$ by \cref{lemma:properties-Priestley-1}. 
    Therefore, \[X = \cl U = \cl\bigcup\{V \in \clopup(X) \mid x \not \in V\},\] and hence $Y \subseteq \bigcup\{V \in \clopup(X) \mid x \not \in V\}$ by \cref{cor:y-pulls-closure-1}. If $Y$ were compact, there would exist a clopen upset $V$ such that $x \not \in V$ and $Y \subseteq V$. Since $X$ is L-spatial, this would imply $x\in X = \cl Y \subseteq V$, a contradiction. Thus, $Y$ is not compact.
\end{proof}

\begin{theorem} \label{thm:shp-iff-stcsp}
Let $X$ be an SL-space and $Y$ its localic part. Then $X$ is a StKL-space iff $Y$ is stably compact.  
\end{theorem}

\begin{proof}
Apply \cref{thm:shp-iff-stlc,lemma:scott-iff-compact}. 
\end{proof}

\begin{corollary} \label{thm: tpl = stk}
{\SKL} is equivalent to {\SKSp}.
\end{corollary}

\begin{proof}
Apply \cref{thm: st loc comp,thm:shp-iff-stcsp}.
\end{proof}

As a consequence of \cref{thm: st comp = tpl,thm: tpl = stk}, we obtain the following well-known duality for stably compact frames (see \cref{SKFrm-deqv-SKsp}):

\begin{corollary}
{\SKFrm} is dually equivalent to {\SKSp}.
\end{corollary}

\section{Deriving Isbell duality} \label{sec: 6}

A characterization of the well inside relation $\prec$ on a frame $L$ in the language of the Priestley space of $L$ was given in \cite[Sec.~3]{BezhanishviliGabelaiaJibladze2016}. Similar to the notion of the kernel of a clopen upset $U$, which we introduced in \cref{sec: 3}, the well inside relation induces the notion of the regular part of $U$.

\begin{definition}
    Let $X$ be an L-space. For $U,V\in\clopup(X)$ we write $U \prec V$ provided $\downset U \subseteq V$. Define the \emph{regular part} of $U$ as 
    \[
    \reg U = \bigcup\{V \in \clopup \mid V \prec U\}.
    \] 
    When $U = \varphi(a)$, we write $\reg(a)$ for $\reg U$.
\end{definition}

This definition is motivated by the following: 

\begin{lemma} [{\cite[Sec.~3]{BezhanishviliGabelaiaJibladze2016}}] \label{lem: well inside}
   Let $X$ be an L-space. For each $U\in\clopup(X)$ we have 
   \[
   \reg U = X\setminus \downset \upset (X \setminus U).
   \] 
In particular, if $X$ is the Priestley space of a frame $L$, then for $a,b \in L$ we have:
\[
a \prec b \mbox{ iff } 
\varphi(a) \prec \varphi(b) \mbox{ iff } \varphi(a) \subseteq \reg(b).
\] 
\end{lemma}

Consequently, a frame $L$ is regular iff in its Priestley space the regular part of each clopen upset $U$ is dense in $U$ (see \cite[Lem.~3.6]{BezhanishviliGabelaiaJibladze2016}). We will give several equivalent characterizations for this condition 
in \cref{lemma:regularity}. For this we need the following: 

\begin{lemma} \label{cor:containment-reg}
Let $X$ be an L-space, $x \in X$, $Z \subseteq X$, and $U \in \clopup(X)$. 
\begin{enumerate}[ref=\thelemma(\arabic*)]
\item $x \in \reg U$ iff $\downset \upset x \subseteq U$. \label[lemma]{cor:containment-reg-1}
\item $Z \subseteq \reg U$ iff $\downset \upset Z \subseteq U$ \label[lemma]{cor:containment-reg-2}
\end{enumerate}
    
\end{lemma}

\begin{proof}
    (1) By Lemma~\ref{lem: well inside},
    \begin{align*}
            x \in \reg U &\iff x \notin \downset \upset (X\setminus U) \iff \upset x \cap \upset(X\setminus U) = \varnothing \\
            &\iff \downset \upset x \cap (X \setminus U) = \varnothing \iff \downset \upset x \subseteq U.
    \end{align*}

    (2) This follows from (1) since $\downset \upset Z = \bigcup \{ \downset \upset x \mid x \in Z \}$.
\end{proof}

\begin{lemma} \label{lemma:regularity}
    Let $X$ be an L-space, $Y$ the localic part of $X$, and $U \in \clopup(X)$. The following three conditions are equivalent. 
    \begin{enumerate}[ref=\thelemma(\arabic*)]
        \item $U \cap Y \subseteq \reg U$.
        \item for each $y \in U \cap Y$ there are disjoint clopen upsets $V, W$ such that $y \in V$ and $U^c \subseteq W$. \label[lemma]{lemma:regularity-2}
        \item $\downset \upset (U \cap Y) \subseteq U$.
    \end{enumerate}
    Moreover, the following condition implies conditions {\upshape (1)--(3)}.
    \begin{enumerate}[resume]
        \item $\reg U$ is dense in $U$.
    \end{enumerate}
    Furthermore, if $X$ is L-spatial, all four conditions are equivalent.

\end{lemma}
\begin{proof}
    (1)$\Rightarrow$(2) Let $y \in U \cap Y$. By (1), $y \in \reg U$. Hence, there is a clopen upset $V$ such that $y\in V$ and $\downset V \subseteq U$. Since $X$ is an Esakia space, $\downset V$ is a clopen downset. Therefore, $W := (\downset V)^c$ is a clopen upset disjoint from $V$ with $U^c \subseteq W$.

    (2)$\Rightarrow$(3) Let $x \in \downset \upset (U \cap Y)$. Then there is $y \in U \cap Y$ such that $x \in \downset \upset y$. By (2), there are disjoint clopen upsets $V,W$ such that $y \in V$ and $U^c \subseteq W$. Thus, $x \in \downset \upset y \subseteq \downset V \subseteq W^c \subseteq U$.
    
    (3)$\Rightarrow$(1) Apply \cref{cor:containment-reg-2}.
    
    Therefore, conditions (1)--(3) are equivalent.    
    
    (4)$\Rightarrow$(1) Let $y \in U \cap Y$. Then $y \in \cl\reg U$ by (4). Since $\downset y$ is open and $\reg U$ is an upset, we conclude that $y \in \reg U$. 
    
    Let $X$ be L-spatial.
    
    (1)$\Rightarrow$(4) From $U \cap Y \subseteq \reg U$ it follows that $\cl(U \cap Y) \subseteq \cl\reg U$. But $\cl(U \cap Y) = U$ since $X$ is L-spatial. Thus, $\reg U$ is dense in $U$.
\end{proof}
\begin{remark}
    Compare \cref{lemma:regularity-2} to the usual definition of regularity in topological spaces.
\end{remark}

\begin{definition} \label{def: regular}
Let $X$ be an L-space. 
    \begin{enumerate}
        \item A clopen upset $U$ of $X$ is \emph{L-regular} if $\reg U$ is dense in $U$.
        \item $X$ is \emph{L-regular} if all its clopen upsets are L-regular.
        \item $X$ is a \emph{KRL-space} if $X$ is L-compact and L-regular.
\end{enumerate}
\end{definition}

Let {\KRL} be the full subcategory of {\LPries} consisting of KRL-spaces. The next theorem goes back \cite[Sec.~3]{BezhanishviliGabelaiaJibladze2016} (see also \cite[Sec.~3]{PultrSichler1988}).

\begin{theorem} 
Let $L$ be a frame and $X_L$ its Priestley space.
\begin{enumerate}[ref=\thetheorem(\arabic*)]
\item {\upshape ({\cite[Lem.~3.6]{BezhanishviliGabelaiaJibladze2016}})} $L$ is regular iff $X_L$ is L-regular. \label[lemma]{thm: regular is rl}
\item {\upshape ({\cite[Thm.~3.9]{BezhanishviliGabelaiaJibladze2016}})} $L$ is compact regular iff $X_L$ is a KRL-space. \label[lemma]{thm: compact regular}
\end{enumerate}
\end{theorem}

\begin{corollary} \label{KRFrm=RPL}
    {\KRFrm} is dually equivalent to {\KRL}.
\end{corollary}

\begin{proof}
Apply Theorems~\ref{thm: PS} and~\ref{thm: compact regular}.
\end{proof}

\begin{lemma}
    Let $X$ be a KRL-space and $Y$ the localic part of $X$.
     Then $X$ is L-spatial and $Y$ is compact. \label[lemma]{lemma:RPL-is-SL} \label[lemma]{lemma:rpl-implies-compact}
\end{lemma}

\begin{proof}
By \cref{lemma:scott-iff-compact}, it is sufficient to show that $X$ is L-spatial, for which, by \cref{lemma:properties-Priestley-0}, it is enough to show that $U \cap V^c \neq \varnothing$ implies $U \cap V^c \cap Y \neq \varnothing$ for all $U,V \in \clopup(X)$. Since $X$ is L-regular, $U = \cl\reg U$. Therefore, $U \cap V^c \neq \varnothing$ implies $\reg U \cap V^c \neq \varnothing$. Let $z \in \reg U \cap V^c$. Because $z \in \reg U$, there is a clopen upset $W$ containing $z$ such that $\downset W \subseteq U$. By \cref{lemma:properties-Priestley-3}, there is $y \in \min(\downset W)$ with $y \leq z$. Consequently, $y \in \downset W \subseteq U$ and $y \in V^c$ since $V^c$ is a downset. Moreover, $y \in Y$ because $y \in \min(X)$ and $\min(X) \subseteq Y$ since $X$ is L-compact.
\end{proof}

\begin{remark}
\cref{lemma:RPL-is-SL} provides an alternative proof of the well-known fact that each compact regular frame is spatial (see, e.g., \cite[p.~90]{Johnstone1982}).
\end{remark}

\begin{theorem} \label{thm: RPL and KHaus}
Let $X$ be an SL-space and $Y$ the localic part of $X$. Then $X$ is a KRL-space iff $Y$ is compact Hausdorff.
\end{theorem}

\begin{proof}
    Let $X$ be a KRL-space. By \cref{lemma:rpl-implies-compact}, $Y$ is compact. We prove that $Y$ is regular. Let $y \in Y$ and $F$ be a closed subset of $Y$ with $y\notin F$. Then $Y\setminus F$ is an open subset of $Y$ containing $y$. Therefore, there is a clopen upset $U$ of $X$ with $U\cap Y=Y\setminus F$. 
    Since $X$ is L-regular, $U$ is L-regular. Therefore, by the implication (4)$\Rightarrow$(2) in \cref{lemma:regularity}, there exist disjoint clopen upsets $V, W$ such that $y \in V$ and $U^c \subseteq W$. Thus, $V \cap Y$ and $W \cap Y$ are disjoint open subsets of $Y$ such that $y \in V \cap Y$ and $F=Y \setminus U \subseteq W \cap Y$. This implies that $Y$ is regular. 
    Consequently, $Y$ is compact Hausdorff.
    
    Conversely, let $Y$ be compact Hausdorff. Since $X$ is an SL-space, $X$ is L-compact by \cref{lemma:scott-iff-compact}. To see that $X$ is L-regular, let $U$ be a clopen upset of $X$. Suppose $y \in U \cap Y$. Since $Y$ is regular, $Y \setminus U$ is closed in $Y$, and $y \not \in Y \setminus U$, there exist clopen upsets $V, W$ such that $V \cap W \cap Y = \varnothing$, $y \in V$, and $Y \setminus U \subseteq W \cap Y$. Therefore, since $X$ is L-spatial,
    \[
    V \cap W = \cl(V \cap Y) \cap \cl(W \cap Y) = \cl(V \cap W \cap Y) = \cl \varnothing = \varnothing,
    \] 
    where the second equality follows from Lemma~\ref{lem:cl-respects-meet-2}. 
    Moreover, $Y \setminus U \subseteq W \cap Y$ implies $U^c \subseteq W$ because $X$ is L-spatial. Thus, $U$ is L-regular by \cref{lemma:regularity}. This finishes the proof that $X$ is a KRL-space.
\end{proof}

\begin{corollary} \label{RPL=KHaus}
    {\KRL} is  equivalent to {\KHaus}.
\end{corollary}

\begin{proof}
Apply \cref{cor: sl=sob}, \cref{lemma:rpl-implies-compact}, and \cref{thm: RPL and KHaus}.
\end{proof}

We can now derive Isbell duality from \cref{KRFrm=RPL,RPL=KHaus}:

\begin{corollary}
    {\KRFrm} is dually equivalent to {\KHaus}.
\end{corollary}

We next compare $\ker$ with $\reg$. This is reminiscent of the comparison between compact and complemented elements in frames.

\begin{lemma} \label{lemma:compact-regular-implies-reg=ker}
    Let $X$ be an L-space.
    \begin{enumerate}[label={\upshape(\arabic*)},ref=\thelemma(\arabic*)]
        \item If $U\in\clopup(X)$ is L-regular, 
        then $\ker U \subseteq \reg U$. \label[lemma]{lemma:ker-subset-reg}
        \item $X$ is L-compact iff $\reg U \subseteq \ker U$ for every $U\in\clopup(X)$.  \label[lemma]{lemma:reg-subset-ker}
        \item If $X$ is a KRL-space, then $\reg U = \ker U$ for every $U\in\clopup(X)$. \label[lemma]{lemma:reg=ker}
    \end{enumerate}
\end{lemma}

\begin{proof}
(1) Since $U$ is L-regular, $\cl\reg U = U$. Therefore, $\ker U \subseteq \reg U$ by \cref{lemma:kernel-adjoint-of-closure}. 

(2) First suppose that $X$ is L-compact and $U\in\clopup(X)$. We show that $V \prec U$ implies $V \ll U$ for every $V\in\clopup(X)$. 
Let $U \subseteq \cl (W)$ for some open upset $W$. Then $U \cap Y \subseteq W$ by \cref{cor:y-pulls-closure-1}.
Moreover, since
$\downset V \subseteq U$, 
we have $\min (\downset V) \subseteq U$. Therefore, $\min(\downset V) \subseteq U \cap Y$ because $X$ is L-compact.
Thus, $V \subseteq \upset \min(\downset V) \subseteq \upset(U \cap Y) \subseteq W$. Consequently, $V \ll U$, and hence $\reg U \subseteq \ker U$.

Conversely, since $\reg X = X$ (see Lemma~\ref{lem: well inside}),
$\reg X \subseteq \ker X$ implies $\ker X = X$, so $X$ is L-compact by \cref{lemma:compactness}.

(3) This follows from (1) and (2).
\end{proof}

For the next lemma we recall from \cref{sec: frm} that a subset of a poset is a biset if it is both an upset and a downset.

\begin{lemma}
    Let $X$ be an L-space and $Y$ the localic part of $X$.
    \begin{enumerate}[ref=\thelemma(\arabic*)]
        \item If $X$ is L-compact, then each closed biset is a Scott upset. \label[lemma]{lemma:tightly-packed-implies-biset-is-scott}
        \item If $X$ is L-regular, then each Scott upset is a biset. \label[lemma]{lemma:regular-implies-scott-is-biset}
        \item If $X$ is L-regular, then $Y \subseteq \min(X)$. \label[lemma]{lemma:Y-subset-Min}
        \item If $X$ is a KRL-space, then closed bisets are exactly Scott upsets.
        \item If $X$ is a KRL-space, then $\min(X) = Y$. Consequently, $\min(\downset F) = \downset F \cap Y$ for every $F \subseteq X$. \label[lemma]{cor:regularly-packed-implies-min=Y}
    \end{enumerate}
\end{lemma}

\begin{proof}
    (1) Since $X$ is L-compact, $\min(X) \subseteq Y$. Therefore, for each closed biset $F$, we have $\min(F) \subseteq \min(X) \subseteq Y$. Thus, $F$ is a Scott upset.
    
    (2) Suppose $F$ is a Scott upset. Let $x \in \downset F$. Then there is $z \in F$ with $x \leq z$. Since $F$ is a Scott upset, there is $y \in F \cap Y$ with $y \leq z$. 
    If $y\not\leq x$, then there is a clopen upset $U$ with $y \in U$ and $x \not \in U$.  Since $X$ is L-regular, we have $\downset\upset(U \cap Y) \subseteq U$ by the implication (4)$\Rightarrow$(3) in Lemma~\ref{lemma:regularity}. Therefore, $\downset \upset y \subseteq U$, so $x \in U$, a contradiction. Thus, we must have $y \leq x$, so $x\in F$, and hence $F$ is a biset.
    
    (3) Let $y \in Y$. Then $\upset y$ is a Scott upset, so $\upset y$ is a downset by (2). Thus, $y \in \min(X)$.
    
    (4) This follows from (1) and (2).
    
    (5) Since $X$ is a KRL-space, $X$ is L-compact, so $\min(X) \subseteq Y$. The reverse inclusion follows from (3). Consequently,
    $\min(\downset F) = \downset F \cap \min(X) = \downset F \cap Y$.
\end{proof} 

\begin{remark}
        The frame-theoretic reading of the first part of \cref{cor:regularly-packed-implies-min=Y} is that in a compact regular frame the minimal prime filters are exactly the completely prime filters. This was first observed in \cite[Lem.~5.2,~5.3]{BezhanishviliGabelaiaJibladze2016}.
\end{remark}

\begin{theorem} \label{thm: RPL implies TPL}
Each KRL-space is a StKL-space.
\end{theorem}

\begin{proof}
Let $X$ be a KRL-space. We first prove that $X$ is a CL-space. Let $U$ be a clopen upset of $X$. By \cref{lemma:reg=ker}, $\reg U = \ker U$, so $\cl\ker U = \cl\reg U = U$ since $U$ is L-regular. Therefore, $X$ is a CL-space.

Next let $U, V$ be clopen upsets of $X$. For each clopen upset $W$, by \cref{lemma:reg=ker}, we have
\begin{align*}
W \subseteq \ker U \cap \ker V
&\iff W \subseteq \reg U \cap \reg V
\iff W \subseteq \reg U, \reg V\\
&\iff \downset W \subseteq U \cap V
\iff W \subseteq \reg (U \cap V) \\
&\iff W \subseteq \ker (U \cap V).
\end{align*}
Therefore, $\ker U \cap \ker V = \ker (U \cap V)$, and so $X$ is kernel-stable. Since $X$ a CL-space, $X$ is Scott-stable by \cref{lemma:kernel-stable-Scott-2}. Also, because $X$ is a KRL-space, $X$ is L-compact. Consequently, $X$ is a StKL-space.
\end{proof}

\begin{theorem}
Let $f : X_1 \to X_2$ be an L-morphism between L-spaces.
\begin{enumerate}[ref=\thetheorem(\arabic*)]
    \item $f^{-1}(\reg U) \subseteq \reg f^{-1}(U)$ for each clopen upset $U$ of $X_2$.
    \item If $X_1$ is L-compact and $X_2$ is L-regular, then $f$ is proper.  \label[theorem]{thm: proper}
\end{enumerate}
\end{theorem}

\begin{proof}
    (1) Suppose $x \in f^{-1}(\reg U)$. Then $f(x) \in \reg U$. Therefore, $\downset \upset f(x) \subseteq U$ by \cref{cor:containment-reg-1}. Since $f$ is order-preserving, we obtain $f(\downset \upset x) \subseteq U$. Thus, $\downset \upset x \subseteq f^{-1}(U)$, and so $x \in \reg f^{-1}(U)$ by \cref{cor:containment-reg-1}.
    
    (2) Let $U$ be a clopen upset of $X_2$. Since $X_2$ is L-regular, $U$ is L-regular. Therefore, since $X_1$ is L-compact, by (1) and \cref{lemma:compact-regular-implies-reg=ker},
    we have
    \[
        f^{-1}(\ker U) \subseteq f^{-1}(\reg U) \subseteq \reg f^{-1}(U) \subseteq \ker f^{-1}(U).
    \] Thus, $f$ is proper.
\end{proof}

\begin{remark}
\cref{thm: proper} corresponds to the well-known fact that every frame homomorphism from a compact frame to a regular frame is proper.
\end{remark}

Putting Theorems~\ref{thm: RPL implies TPL} and~\ref{thm: proper} together yields:

\begin{corollary}
    {\KRL} is a full subcategory of {\SKL}. 
\end{corollary}

We thus arrive at the following diagram, where we use the same notation as in \cref{diagram: introduction}. In addition, the 
arrow ($\longleftrightarrow$) represents an equivalence of categories. For an overview of the introduced categories of Priestley spaces see \cref{table: L-spaces}. The corresponding  categories of frames and spaces are described in \cref{table:frames,table:spaces}.

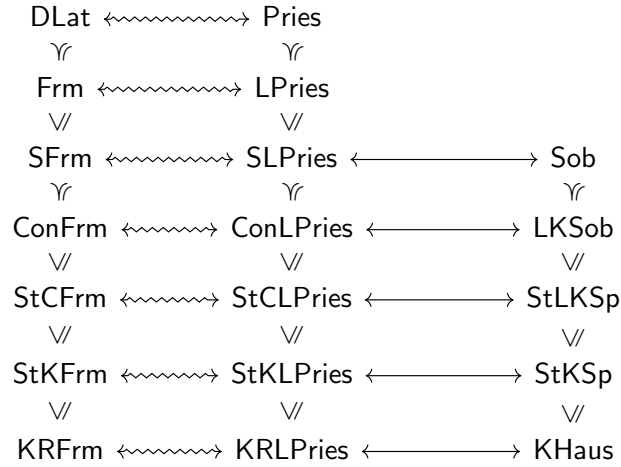
\begin{figure}[H] 
\centering
\begin{tikzcd}[ampersand replacement=\&, column sep={8em,between origins}, row sep=1em]
    \cat{DLat} \ardual \& \cat{Pries}\\
    \Frm \ardual \arsemirestrict \& \LPries \arsemirestrict \\
    \SFrm \ardual \arrestrict \& \SL \arrestrict \areqv \& \Sob \\
    \CFrm \arsemirestrict \ardual \& \CL \areqv \arsemirestrict \& \LCSob \arsemirestrict \\
    \SCFrm  \ardual\arrestrict \& \SCL \areqv\arrestrict \& \SLCSp \arrestrict\\
    \SKFrm  \ardual\arrestrict \& \SKL \areqv\arrestrict\& \SKSp\arrestrict\\
    \KRFrm  \ardual \arrestrict \& \KRL \areqv \arrestrict\& \KHaus\arrestrict
\end{tikzcd}
\caption{Equivalences and dual equivalences between various categories of frames, L-spaces, and sober spaces.\label{diagram 2}}
\end{figure}

\begin{table}[H]
\centering
\begin{tabular}{llll}
    \toprule
    \bf Category & \bf Objects & \bf Morphisms \\
    \midrule
    \LPries & L-spaces  (Def.~\ref{def: L}) & L-morphisms (Def.~\ref{def: L})\\
    \SL & L-spatial L-spaces  (Def.~\ref{def: SL}) & L-morphisms\\
   \CL & continuous L-spaces (Def.~\ref{def: packed}) & proper L-morphisms (Def.~\ref{def: proper L-morphism})\\
   \SCL & stably continuous L-spaces (Def.~\ref{def: StCL}) & proper L-morphisms \\
   \SKL & stably L-compact L-spaces (Def.~\ref{def: TPL}) & proper L-morphisms \\
   \KRL & L-compact L-regular L-spaces (Def.~\ref{def: regular})  & L-morphisms\\
    \bottomrule
\end{tabular}
\caption{Categories of L-spaces.\label{table: L-spaces}}
\end{table}

We conclude the paper by proving that every L-morphism $f : X_1 \to X_2$ from an L-compact L-space to an L-regular L-space satisfies $\downset f(x) = f(\downset x)$ for each $x \in X_1$. This result was first proved in \cite[Cor.~4.3]{BezhanishviliGabelaiaJibladze2016} utilizing that each frame homomorphism from a regular frame to a compact frame is closed (see \cite[Lem.~4.1]{BezhanishviliGabelaiaJibladze2016}).
We give alternative proofs of both Lemma~4.1 and Corollary~4.3 of \cite{BezhanishviliGabelaiaJibladze2016} using the language of Priestley spaces.
We start by the following two lemmas.

\begin{lemma}[{\cite[Rem.~3.7]{BezhanishviliGabelaiaJibladze2016}}] \label{lemma:equal min}
    Let $X$ be an L-regular L-space and $Y$ its localic part. Then $\min(D)= \min(D')$ implies $D = D'$ for all clopen downsets $D, D'$ of $X$.
\end{lemma}

\begin{lemma}
    Let $X_1,X_2$ be L-spaces and $Y_1,Y_2$ their respective localic parts. If ${f : X_1 \to X_2}$ is an L-morphism and $D$ is a clopen downset of $X_1$, then
    \begin{enumerate}[ref=\thelemma(\arabic*)]
\item $\downset f(D)$ is clopen.  \label[lemma]{lemma: downset-fD-clopen}
    \end{enumerate}
    If in addition $X_1$ is L-compact and $X_2$ is L-regular, then
    \begin{enumerate}[resume,ref=\thelemma(\arabic*)]
        \item $\min \downset f(D) \subseteq  f(D) \cap Y_2$. \label[lemma]{lemma:same-min}
        \item $\downset f(D) = f(D)$.
        \item $f(D)$ is a clopen downset. \label[lemma]{lemma:f(D)}
    \end{enumerate}
\end{lemma}

\begin{proof}
    (1) Since $f$ is a closed map, $\downset f(D)$ is closed, and hence $U := (\downset f(D))^c$ is an open upset. Because $D \cap f^{-1}(U) = \varnothing$ and $f$ is an L-morphism, $D \cap f^{-1}(\cl U) = D \cap \cl f^{-1}(U) = \varnothing$. Therefore, $f(D) \cap \cl U = \varnothing$, and so $\downset f(D) \cap \cl U = \varnothing$ since $\cl U$ is an upset. Consequently, $\downset f(D) = (\cl U)^c$, and hence $\downset f(D)$ is clopen (because $\cl U$ is clopen).
    
    (2) Suppose $z \in \min \downset f(D)$. Then $z \le f(x)$ for some $x \in D$. Since $D$ is closed, there is $y \in \min D$ with $y\le x$ (see \cref{lemma:properties-Priestley-3}). Because $X_1$ is L-compact, $\min D \subseteq \min X_1 \subseteq Y_1$. Therefore, $y\in Y_1$, and so $f(y) \in Y_2$ by \cref{lem: f restricts}. But then $\upset f(y)$ is a Scott upset. Thus, since $X_2$ is L-regular, $\upset f(y)$ is a biset by \cref{lemma:regular-implies-scott-is-biset}. Consequently, $f(x) \in \upset f(y)$ implies $\downset f(x) \subseteq \upset f(y)$. Therefore, $z \in \upset f(y)$, so $f(y) \leq z$. But then $f(y) = z$ by the minimality of $z$. Thus, $z \in  f(D) \cap Y_2$.
    
    (3) Clearly $f(D) \subseteq \downset f(D)$. To see the reverse inclusion, since $f(D)$ is a closed subset of $X_2$, it is sufficient to show that $\downset f(D)$ is the closure of $\min \downset f(D)$ because the latter is contained in $f(D)$ by (2).
Suppose otherwise. Since $\downset f(D)$ is clopen by (1), $\downset f(D) \setminus \cl \min \downset f(D)$ is a nonempty open set. Therefore, by \cref{lemma:properties-Priestley-0}, there are $U,V \in \clopup(X_2)$ such that
$\varnothing \neq U \cap V^c \subseteq \downset f(D)$ and $U \cap V^c \cap\cl\min \downset f(D)=\varnothing$. Let $A = \downset (U \cap V^c)$ and $B = A \setminus (U \cap V^c)$. Then $A\ne B$. Clearly $A$ is a clopen downset and $B$ is clopen. We show that $B$ is also a downset and $\min A=\min B$. We first show that $B$ is a downset. Let $x \in B$ and $y \leq x$. If $y \not \in B$ then $y \in U \cap V^c$. Since $x \in \downset(U \cap V^c)$, there exists $z \in U \cap V^c$ such that $x \leq z$. Because $U$ is an upset, from $y \in U$ and $y \leq x$ it follows that $x \in U$. Since $V^c$ is a downset, from $z \in V^c$ and $x \leq z$ it follows that $x \in V^c$. Therefore, $x \in U \cap V^c$, and so $x \not \in B$, a contradiction. Thus, $B$ is a downset.

We next show that $\min A = \min B$. Clearly $\min B\subseteq\min A$. If $x \in \min A$, then $x\in \min \downset f(D)$, so $x \not \in U$. Therefore, $x \in \min B$. Thus, $A$ and $B$ are clopen downsets such that $A \ne B$ but $\min A=\min B$. 
This contradicts \cref{lemma:equal min} because $X_2$ is L-regular.

(4) Apply (1) and (3).
\end{proof}

Consequently, 
the same argument as in the proof of \cite[Cor.~4.3]{BezhanishviliGabelaiaJibladze2016} yields:

\begin{theorem}
    Let $f : X_1 \to X_2$ be a proper L-morphism between an L-compact L-space $X_1$ and an L-regular L-space $X_2$. Then $f(\downset x) = \downset f(x)$ for each $x \in X_1$.
\end{theorem}  

We recall that a frame homomorphism $h : L \to M$ is \emph{closed} if 
$
    r(h(a) \vee b) \leq a \vee r(b)
$
for all $a \in L$ and $b \in M$, where $r : M \to L$ is the right adjoint of $h$. We close by an alternative proof of \cite[Lem.~4.1]{BezhanishviliGabelaiaJibladze2016}.

\begin{theorem}
    If $h : L \to M$ is a frame homomorphism from a regular frame $L$ to a compact frame $M$, then $h$ is closed.
\end{theorem}

\begin{proof}
    Let $f : X_M \to X_L$ be the dual L-morphism between the Priestley spaces of $M$ and $L$, respectively. Suppose $a \in L$ and $b \in M$. Since $X_L$ is L-regular by \cref{thm: regular is rl}, it suffices to show that
    \[
        \reg r(h(a) \vee b) \subseteq \varphi(a) \cup \varphi(r(b)).
    \]
    Let $d\in M$. Since $r$ is right adjoint to $h$, we have $r(d) = \bigvee\{c\in L \mid h(c)\le d\}$. Therefore, since $\varphi(h(c))= f^{-1}(\varphi(c))$, by \cref{lemma:joins-in-priestley}, we have
    \begin{align*}
        \varphi(r(d))
        &= \varphi\left(\bigvee\{c\in L \mid h(c)\le d\}\right)\\
        &= \cl \left(\bigcup\{\varphi(c) \mid f^{-1}(\varphi(c)) \subseteq \varphi(d)\}\right)\\
        &= \cl \left(\bigcup\{\varphi(c) \mid \varphi(c) \subseteq X_L \setminus f(X_M \setminus \varphi(d))\}\right)\\
        &= X_L \setminus f(X_M \setminus \varphi(d)),
    \end{align*}
    where the last equality follows from \cref{lemma:f(D)}. 
    Let $x \in \reg r(h(a) \vee b)$. We show that $x\in \varphi(a) \cup \varphi(r(b))$. By \cref{cor:containment-reg-1},
    \begin{align*}
        x \in \reg r(h(a) \vee b) 
        \iff& \downset\upset x \subseteq X_L \setminus f\big(X_M \setminus (f^{-1}\varphi(a) \cup \varphi(b))\big)\\
        \iff& f^{-1}(\downset \upset x)\subseteq f^{-1}\varphi(a) \cup \varphi(b).
    \end{align*}
     Suppose $x \notin \varphi(r(b))$. 
    Then $x\notin X_L \setminus f(X_M \setminus \varphi(b))$, so $f^{-1}(x) \not \subseteq \varphi(b)$. Therefore, there is 
    $z \in X_M$ such that $f(z)=x$ and $z \not \in \varphi(b)$. From $f(z)=x$ it follows that $z \in f^{-1}(\downset \upset x) \subseteq f^{-1}\varphi(a) \cup \varphi(b)$. Hence, from $z \not \in \varphi(b)$ it follows that $z\in f^{-1}\varphi(a)$, and so $x = f(z) \in \varphi(a)$, concluding the proof.
\end{proof}

\section*{Acknowledgements}
We would like to thank the referees for careful reading and useful suggestions, which have improved the paper.

\bibliographystyle{abbrv}
\bibliography{ref}

\begin{thebibliography}{10}

\bibitem{AlechinaMendlerDePaiva2001}
N.~Alechina, M.~Mendler, V.~de~Paiva, and E.~Ritter.
\newblock Categorical and {K}ripke semantics for constructive {S}4 modal logic.
\newblock In {\em Computer science logic ({P}aris, 2001)}, volume 2142 of {\em
  Lecture Notes in Comput. Sci.}, pages 292--307. Springer, Berlin, 2001.

\bibitem{ArtemovProtopopescu2016}
S.~Artemov and T.~Protopopescu.
\newblock Intuitionistic epistemic logic.
\newblock {\em Rev. Symb. Log.}, 9(2):266--298, 2016.

\bibitem{AvilaBezhanishviliMorandiZaldivar2020}
F.~\'{A}vila, G.~Bezhanishvili, P.~J. Morandi, and A.~Zald\'{\i}var.
\newblock When is the frame of nuclei spatial: a new approach.
\newblock {\em J. Pure Appl. Algebra}, 224(7):106302, 20, 2020.

\bibitem{AvilaBezhanishviliMorandiZaldivar2021}
F.~\'{A}vila, G.~Bezhanishvili, P.~J. Morandi, and A.~Zald\'{\i}var.
\newblock The frame of nuclei on an {A}lexandroff space.
\newblock {\em Order}, 38(1):67--78, 2021.

\bibitem{Banaschewski1981}
B.~Banaschewski.
\newblock Coherent frames.
\newblock In {\em Proceedings of the Conference on Topological and Categorical
  Aspects of Continuous Lattices, Lecture Notes in Math., Vol. 871}, pages
  1--11. Springer-Verlag, Berlin, 1981.

\bibitem{BanaschweskiMulvey1980}
B.~Banaschewski and C.~J. Mulvey.
\newblock Stone-\v{C}ech compactification of locales. {I}.
\newblock {\em Houston J. Math.}, 6(3):301--312, 1980.

\bibitem{BezhanishviliBezhanishvili2008}
G.~Bezhanishvili and N.~Bezhanishvili.
\newblock Profinite {H}eyting algebras.
\newblock {\em Order}, 25(3):211--227, 2008.

\bibitem{BezhanishviliGabelaiaJibladze2013}
G.~Bezhanishvili, D.~Gabelaia, and M.~Jibladze.
\newblock Funayama's theorem revisited.
\newblock {\em Algebra Universalis}, 70(3):271--286, 2013.

\bibitem{BezhanishviliGabelaiaJibladze2016}
G.~Bezhanishvili, D.~Gabelaia, and M.~Jibladze.
\newblock Spectra of compact regular frames.
\newblock {\em Theory Appl. Categ.}, 31:Paper No. 12, 365--383, 2016.

\bibitem{BezhGhilardi2007}
G.~Bezhanishvili and S.~Ghilardi.
\newblock An algebraic approach to subframe logics. {I}ntuitionistic case.
\newblock {\em Ann. Pure Appl. Logic}, 147(1-2):84--100, 2007.

\bibitem{BezhanishviliHolliday2019}
G.~Bezhanishvili and W.~H. Holliday.
\newblock A semantic hierarchy for intuitionistic logic.
\newblock {\em Indag. Math. (N.S.)}, 30(3):403--469, 2019.

\bibitem{BezhanishviliMelzer2022}
G.~Bezhanishvili and S.~Melzer.
\newblock Hofmann--{M}islove through the lenses of {P}riestley.
\newblock {\em Semigroup Forum}, 105(3):825--833, 2022.

\bibitem{Cornish1975}
W.~H. Cornish.
\newblock On {H}. {P}riestley's dual of the category of bounded distributive
  lattices.
\newblock {\em Mat. Vesnik}, 12(27)(4):329--332, 1975.

\bibitem{DickmannSchwartzTressl2019}
M.~Dickmann, N.~Schwartz, and M.~Tressl.
\newblock {\em Spectral spaces}, volume~35 of {\em New Mathematical
  Monographs}.
\newblock Cambridge University Press, Cambridge, 2019.

\bibitem{DowkerPapert1966}
C.~H. Dowker and D.~Papert.
\newblock Quotient frames and subspaces.
\newblock {\em Proc. London Math. Soc. (3)}, 16:275--296, 1966.

\bibitem{Engelking1989}
R.~Engelking.
\newblock {\em General topology}, volume~6 of {\em Sigma Series in Pure
  Mathematics}.
\newblock Heldermann Verlag, Berlin, second edition, 1989.
\newblock Translated from the Polish by the author.

\bibitem{Esakia1974}
L.~L. Esakia.
\newblock Topological {K}ripke models.
\newblock {\em Soviet Math. Dokl.}, 15:147--151, 1974.

\bibitem{Esakia2019}
L.~L. Esakia.
\newblock {\em Heyting algebras}, volume~50 of {\em Trends in Logic---Studia
  Logica Library}.
\newblock Springer, Cham, 2019.
\newblock Edited by G. Bezhanishvili and W. H. Holliday, Translated from the
  Russian by A. Evseev.

\bibitem{FairtloughMendler1995}
M.~Fairtlough and M.~Mendler.
\newblock An intuitionistic modal logic with applications to the formal
  verification of hardware.
\newblock In {\em Computer science logic ({K}azimierz, 1994)}, volume 933 of
  {\em Lecture Notes in Comput. Sci.}, pages 354--368. Springer, Berlin, 1995.

\bibitem{FairtloughMendler1997}
M.~Fairtlough and M.~Mendler.
\newblock Propositional lax logic.
\newblock {\em Inform. and Comput.}, 137(1):1--33, 1997.

\bibitem{GargAbadi2008}
D.~Garg and M.~Abadi.
\newblock A modal deconstruction of access control logics.
\newblock In {\em Foundations of Software Science and Computational
  Structures}, volume 4962 of {\em Lecture Notes in Comput. Sci.}, pages
  216--230. Springer, Berlin, 2008.

\bibitem{Compendium2003}
G.~Gierz, K.~H. Hofmann, K.~Keimel, J.~D. Lawson, M.~Mislove, and D.~S. Scott.
\newblock {\em Continuous lattices and domains}, volume~93 of {\em Encyclopedia
  of Mathematics and its Applications}.
\newblock Cambridge University Press, Cambridge, 2003.

\bibitem{GierzKeimel1977}
G.~Gierz and K.~Keimel.
\newblock A lemma on primes appearing in algebra and analysis.
\newblock {\em Houston J. Math.}, 3(2):207--224, 1977.

\bibitem{Goldblatt1981}
R.~I. Goldblatt.
\newblock Grothendieck topology as geometric modality.
\newblock {\em Z. Math. Logik Grundlagen Math.}, 27(6):495--529, 1981.

\bibitem{Goldblatt2011}
R.~I. Goldblatt.
\newblock Cover semantics for quantified lax logic.
\newblock {\em J. Logic Comput.}, 21(6):1035--1063, 2011.

\bibitem{HofmannLawson1978}
K.~H. Hofmann and J.~D. Lawson.
\newblock The spectral theory of distributive continuous lattices.
\newblock {\em Trans. Amer. Math. Soc.}, 246:285--310, 1978.

\bibitem{HofmannMislove1981}
K.~H. Hofmann and M.~W. Mislove.
\newblock Local compactness and continuous lattices.
\newblock In {\em Continuous lattices}, pages 209--248. Springer, 1981.

\bibitem{Isbell1972}
J.~R. Isbell.
\newblock Atomless parts of spaces.
\newblock {\em Math. Scand.}, 31:5--32, 1972.

\bibitem{Johnstone1981}
P.~T. Johnstone.
\newblock The {G}leason cover of a topos. {II}.
\newblock {\em J. Pure Appl. Algebra}, 22(3):229--247, 1981.

\bibitem{Johnstone1982}
P.~T. Johnstone.
\newblock {\em Stone spaces}, volume~3 of {\em Cambridge Studies in Advanced
  Mathematics}.
\newblock Cambridge University Press, Cambridge, 1982.

\bibitem{MacLane1998}
S.~Mac~Lane.
\newblock {\em Categories for the working mathematician}, volume~5 of {\em
  Graduate Texts in Mathematics}.
\newblock Springer-Verlag, New York, second edition, 1998.

\bibitem{PicadoPultr2012}
J.~Picado and A.~Pultr.
\newblock {\em Frames and locales}.
\newblock Frontiers in Mathematics. Birkh\"{a}user/Springer Basel AG, Basel,
  2012.

\bibitem{Priestley1970}
H.~A. Priestley.
\newblock Representation of distributive lattices by means of ordered {S}tone
  spaces.
\newblock {\em Bull. London Math. Soc.}, 2:186--190, 1970.

\bibitem{Priestley1972}
H.~A. Priestley.
\newblock Ordered topological spaces and the representation of distributive
  lattices.
\newblock {\em Proc. London Math. Soc. (3)}, 24:507--530, 1972.

\bibitem{Priestley1984}
H.~A. Priestley.
\newblock Ordered sets and duality for distributive lattices.
\newblock In {\em Orders: description and roles ({L}'{A}rbresle, 1982)},
  volume~99 of {\em North-Holland Math. Stud.}, pages 39--60. North-Holland,
  Amsterdam, 1984.

\bibitem{PultrSichler1988}
A.~Pultr and J.~Sichler.
\newblock Frames in {P}riestley's duality.
\newblock {\em Cahiers Topologie G\'{e}om. Diff\'{e}rentielle Cat\'{e}g.},
  29(3):193--202, 1988.

\bibitem{PultrSichler2000}
A.~Pultr and J.~Sichler.
\newblock A {P}riestley view of spatialization of frames.
\newblock {\em Cahiers Topologie G\'{e}om. Diff\'{e}rentielle Cat\'{e}g.},
  41(3):225--238, 2000.

\bibitem{Schwartz2013}
N.~Schwartz.
\newblock Locales as spectral spaces.
\newblock {\em Algebra Universalis}, 70(1):1--42, 2013.

\bibitem{Simmons1982}
H.~Simmons.
\newblock A couple of triples.
\newblock {\em Topology Appl.}, 13(2):201--223, 1982.

\bibitem{Vickers1989}
S.~Vickers.
\newblock {\em Topology via logic}, volume~5 of {\em Cambridge Tracts in
  Theoretical Computer Science}.
\newblock Cambridge University Press, Cambridge, 1989.

\end{thebibliography}
\end{document}